\documentclass{article}
\usepackage[cp1251]{inputenc}
\usepackage[english]{babel}
\usepackage{amsmath}
\usepackage{amsthm}
\usepackage{mathtext}
\usepackage{amssymb}
\usepackage{amscd}
\usepackage{graphics}
\usepackage{euscript}
\usepackage{bbm}
\theoremstyle{definition}
\newtheorem{defin}{Definition}

\newtheorem*{remark}{Remark}
\theoremstyle{plain}

\newtheorem*{theo*}{Theorem}
\newtheorem{prop}{Proposition}
\newtheorem{cor}[prop]{Corollary}
\newtheorem*{prop*}{Proposition}

\DeclareMathOperator{\link}{link} \DeclareMathOperator{\rk}{rk}
\DeclareMathOperator{\relint}{relint}
\DeclareMathOperator{\conv}{conv} \DeclareMathOperator{\aff}{aff}
\DeclareMathOperator{\Tor}{Tor} \DeclareMathOperator{\bideg}{bideg}
\DeclareMathOperator{\C}{H} \DeclareMathOperator{\sr}{s_{\mathbb R}}
\newcommand{\sh}{\mathop{\rm \#}\limits}
\newcommand {\ib}[1]{\textit{\textbf{#1}}}

\begin{document}
\title{\bf\large{Buchstaber Invariant of Simple Polytopes.}}
\author{Nickolai Erokhovets.}
\date{}
\maketitle
\begin{abstract} In this paper we study a new combinatorial
invariant of simple polytopes, which comes from toric topology. With
each simple $n$-polytope $P$ with $m$ facets we can associate a
moment-angle complex $\mathcal Z_P$ with a canonical action of the
torus $T^m$. Then $s(P)$ is the maximal dimension of a toric
subgroup that acts freely on $\mathcal Z_P$. The problem stated by
Victor\,M. Buchstaber is to find a simple combinatorial description
of an $s$-number. We describe the main properties of $s(P)$ and
study the properties of simple $n$-polytopes with $n+3$ facets. In
particular, we find the value of an $s$-number for such polytopes, a
simple formula for their $h$-polynomials and the bigraded cohomology
rings of the corresponding moment-angle complexes.
\end{abstract}

\section{Introduction.}
Let $P^n=\{\ib{x}\in\mathbb R^n\colon A_p\ib{x}+\ib{b}_p\geqslant
0\}$ be a simple polytope and $\mathfrak{F}=\{F_1,\dots,F_m\}$ the
set of facets of $P^n$. For each facet $F_i\in\mathfrak F$ denote by
$T^{F_i}$ the one-dimensional coordinate subgroup of $T^{\mathfrak
F}=T^m$. Then assign to every face $G$ the coordinate subtorus
$T^G=\prod_{F_i\supset G}T^{F_i}\subset T^{\mathfrak F}$. For each
point $\ib{q}\in P$ we denote by $G(\ib{q})$ a unique face
containing $\ib{q}$ in its relative interior.

For any combinatorial simple polytope $P^n$ we introduce the identification
space: $\mathcal Z_P=(T^\mathfrak F\times\nobreak P^n)/{\sim}$\,, where
$(t_1,\ib{p})\sim(t_2,\ib{q})$ if and only if $\ib{p}=\ib{q}$ and
$t_1t_2^{-1}\in T^{G(\ib{p})}$.

It turns out that $\mathcal Z_P$ is a smooth manifold of dimension
$m+n$ with a smooth action of $T^m$ induced by the standard action
of the torus on the first factor. Then $\mathcal Z_P/T^m=P$, and the
stabilizer of a point $[(t,\ib{q})]$ is $T^{G(\ib{q})}$.

\begin{defin}[See \cite{BP}]
\hskip-1mm The {\itshape Buchstaber number} $s(P)$ is the maximal
dimension of a torus subgroup $H\cong T^s$ that acts freely.
\end{defin}

The problem stated by Victor\,M. Buchstaber in 2002 is to find a
simple combinatorial description of $s(P)$.

It is possible to define $s(K)$ for any simplicial complex in such a
way that $s(P)=s(\partial P^*)$, where $\partial P^*$ is a boundary
complex of a dual polytope.

In this article we establish some main properties of the $s$-number,
which we summarize in the following theorem (see \cite{E}). Some
definitions are given below.
\begin{theo*}
The $s$-number satisfies the following properties:

\begin{enumerate}

\item $s(P)\geqslant s(Q)$ if $Q$ is obtained from $P$ by forgetting
one of the inequalities $\ib{a}_i\ib{x}+b_i\geqslant 0$.

\item $s(P)=1$ if and only if $P=\Delta^n$.

\item If $n+2\leqslant m\leqslant\frac{49}{48}n+\frac{83}{48}$,
then\quad $s(C^n(m)^*)=2$. In particular, for each $k\geqslant 2$
there exists a polytope with $m-n=k$ and $s(P)=2$.

\item $s(P)+s(Q)\leqslant s(P\times Q)\leqslant
s(P)+s(Q)+\min\{m_1-n_1-s(P),m_2-n_2-s(Q)\}$.

\item $s(P)+s(Q)\leqslant s(P\sh Q)$, where $P\sh Q$ is a connected sum
    of polytopes along  vertices.

\item If $f:K_1^{n-1}\to K_2^{n-1}$ is a non-degenerate map of simplicial complexes,
then
$$
m_1-s(K_1)\leqslant m_2-s(K_2),
$$
where $m_1$ and $m_2$ are
the numbers of vertices of $K_1$ and $K_2$.

\item $m_F-s(F)\leqslant m-s(P)$, where $F$ is a facet of $P$.

\item $s(P)\geqslant m-\gamma(P)+s(\Delta^{\gamma-1}_{n-1})$, where
$\gamma(P)$ is a chromatic number of $P$ and
$\Delta^{\gamma-1}_{n-1}$ is a $(n-1)$-dimensional skeleton of the
simplex $\Delta^{\gamma-1}$. In particular,
$$
s(P)\geqslant
m-\gamma+\left[\frac{\gamma}{n+1}\right]\geqslant\left[\frac{m}{n+1}\right].
$$

\item \begin{itemize}
\item[a)] $s(\Delta^{m-1}_{n-1})\geqslant 2$ if and only if
$\frac{m}{n+1}\geqslant \frac{3}{2}$;
\item[b)] $s(\Delta^{m-1}_{n-1})\geqslant 3$ if and only if
$$
4m\geqslant\begin{cases}
7(n+1),&m=0\mod 7;\\
7(n+1)+4,&m=1\mod 7;\\
7(n+1)+8,&m=2\mod 7;\\
7(n+1)+5,&m=3\mod 7;\\
7(n+1)+2,&m=4\mod 7;\\
7(n+1)+6,&m=5\mod 7;\\
7(n+1)+3,&m=6\mod 7;\\
\end{cases}
$$
\end{itemize}

\item For a simplicial complex $K$ on the set of vertices
$[m]=\{1,\dots,m\}$ let $\omega$ be a minimal non-simplex, if
$\omega\notin K$, but any proper face $\sigma\subset\omega$ belongs
to $K$. Then if $\omega_1,\dots,\omega_l$ is a collection of minimal
non-simplices such that $\omega_1\cup\dots\cup\omega_l=[m]$, then
$$
m-s(K)\leqslant
\dim\omega_1+\dots+\dim\omega_l=(|\omega_1|-1)+\dots+(|\omega_l|-1).
$$

\item A simple polytope is called flag if any collection of facets
$F_{i_1},\dots,F_{i_l}$ such that any two $F_{i_s}, F_{i_t}$ of them
intersect: $F_{i_s}\cap F_{i_t}\ne \varnothing$ has nonempty
intersection $F_{i_1}\cap\dots\cap F_{i_l}\ne\varnothing$. Then for
a flag polytope $P^n$ with a chromatic number $\gamma$ we have:
$\gamma\leqslant[\frac{m-n}{2}]+n$, thus
$$
s(P)\geqslant
\left\lceil\frac{m-n}{2}\right\rceil+s(\Delta^{\gamma-1}_{n-1})
$$

\item A simple polytope is called $k$-flag, if any collection of facets
$F_{i_1},\dots,F_{i_l}$ such that any $k$ of them have nonempty
intersection has nonempty intersection: $F_{i_1}\cap\dots\cap
F_{i_l}\ne\varnothing$. Then for a $k$-flag polytope $P^n$:
$$
s(P)\geqslant \left\lceil \frac{m-n}{k}\right\rceil-(k-2)n.
$$

\item Let $\mathbb Z[P]=\mathbb Z[v_1,\dots,v_m]/I$,
$I=(v_{j_1}v_{j_2}\dots v_{j_l}:F_{j_1}\cap\dots\cap
F_{j_l}=\varnothing)$ be a face ring of a polytope $P$.
$P_{k_1,\dots,\,k_m}$ is defined as a (combinatorial) simple
polytope with the face ring
\begin{multline*}
\mathbb Z[v^1_1,v^2_1,\dots,v^{k_1}_1,v^1_2,\dots,
v^{k_1}_2,\dots,v^1_m,\dots, v^{k_m}_m]/J, \\
J=(v^1_{j_1}v^2_{j_1}\dots v^{k_{j_1}}_{j_1}v^1_{j_2}v^2_{j_2}\dots
v^{k_{j_2}}_{j_2}\dots v^1_{j_l}v^2_{j_l}\dots v^{k_{j_l}}_{j_l}:
v_{j_1}\dots v_{j_l}\in I).
\end{multline*}
Then $s(P_{k_1,\dots,\,k_m})=s(P)$.

\item It is known (see {\upshape \cite{GB}}) that each simple polytope
    $P^n$ with $m=n+3$ facets can be represented in terms of a regular
    $(2k-1)$-gon $M_{2k-1}$ and a surjective map from $\mathfrak
    {F}=\{F_1,\dots, F_{n+3}\}$ to the set of vertices of $M_{2k-1}$. The
    facets $F_{i_1},\dots,F_{i_n}$ intersect in a vertex if and only if
    the triangle formed by the vertices corresponding to the rest three
    facets contains the center of $M_{2k-1}$.

Then for such a polytope $P_{a_1,\dots,\,a_{2k-1}}$ we have: $s(P)=3$ if
and only if $k\leqslant 4$.

Here $k$ can be expressed in terms of bigraded Betti numbers
$$
2k-1=\sum\limits_{j}\beta^{-1,\,2j}(\mathcal
Z_{P_{a_1,\dots,\,a_{2k-1}}})=\sum\limits_{j}\beta^{-2,\,2j}(\mathcal
Z_{P_{a_1,\dots,\,a_{2k-1}}}).
$$
\item There are two polytopes $P$ and $Q$ with the equal $f$-vectors and
chromatic numbers, but the different $s$-numbers. Nevertheless our
$P$ and $Q$ have different bigraded Betti numbers.

\item If $P$ is obtained from $Q$ by one $i$-flip, $2\leqslant i\leqslant n-1$, then
$|s(P)-s(Q)|\leqslant 1$.
$$
s(P)+1\leqslant s(P\sh \Delta^n)\leqslant s(P)+2.
$$
\end{enumerate}
\end{theo*}

We also investigate the properties of simple polytopes with $n+3$
facets. In particular, we will prove the following fact:
\begin{theo*}
For the polytope $P=P_{a_1,\dots,\,a_{2k-1}}$ the bigraded cohomology ring
$\C^{*,\,*}(\mathcal Z_P)$ is isomorphic to the free abelian group $\mathbb
Z\oplus \mathbb Z^{2k-1}\oplus\mathbb Z^{2k-1}\oplus \mathbb Z$ with the
generators
\begin{center}
$1$, $\bideg 1=(0,0)$;\\
$X_i$, $\bideg X_i=(-1,2(a_i+\dots+a_{i+k-2})), i=1,\dots, 2k-1$;\\
$Y_j$, $\bideg Y_j=(-2,2(a_j+\dots+a_{j+k-1})), j=1,\dots,2k-1$;\\
$Z$, $\bideg Z=(-3,2(n+3))$.
\end{center}
For $k\geqslant 3$
$$
X_i\cdot X_j=0\qquad X_i\cdot Y_j=\delta_{i+k-1,\,j}Z\qquad Y_i\cdot
Y_j=0,
$$
and for $k=2$
$$
X_i^2=0,\qquad X_iX_{i+1}=-X_{i+1}X_i=Y_i,\qquad X_1X_2X_3=Z.
$$
\end{theo*}

This in not surprising, since $\mathcal
Z_{P_{a_1,\,a_2,\,a_3}}=S^{2a_1-1}\times S^{2a_2-1}\times
S^{2a_3-1}$, and according to the results by Lopez de Medrano
\cite{LM} for $k\geqslant 3$ the manifold
$\mathcal{Z}_{P_{a_1,\dots,\,a_{2k-1}}}$ is homeomorphic to
$$
\sh\limits_{i=1}^{2k-1}S^{2\varphi_i-1}\times S^{2\psi_{i+k-1}-2},
$$
where $\varphi_i=a_i+\dots+a_{i+k-2}$, $\psi_j=a_j+\dots+a_{j+k-1}$,
and indices are taken modulo $2k-1$. See also \cite{BM}. Our result
describes additionally a bigraded structure in the cohomology ring
of the moment-angle manifold $\mathcal
Z_{P_{a_1,\dots,\,a_{2k-1}}}$.

We also describe some properties of the construction $P\to
P_{k_1,\dots,\,k_m}$ (see \cite{BBCG,BBCG2,GLM}).

The author is grateful to Victor\,M. Buchstaber for the formulation of the
problem and for his permanent attention to this work, to Taras Panov for
useful advises and comments, and to Yukiko Fukukawa and Mikiya Masuda for
sending their paper ``Buchstaber Invariants of Skeletons of a Simplex''.
\section{Some facts.}   It is known that
\begin{enumerate}
\item  $1\leqslant s(P)\leqslant m-n$. $S(P)=m-n$ if and only if there exists a
{\itshape characteristic map} from the set of facets $\mathfrak F$
to $\mathbb Z^n$ such that for every vertex $v=F_{i_1}\cap
F_{i_2}\dots\cap F_{i_n}$ the vectors corresponding to the facets
$F_{i_1},\dots, F_{i_n}$ form a basis of $\mathbb Z^n$. In this case
$\mathcal Z_{P}/T^{m-n}=M^{2n}$ is a {\itshape quasitoric manifold}.

\item Any polygon has a characteristic map, since we can assign to
it's edges the vectors $\begin{pmatrix}0\\1\end{pmatrix}$,
$\begin{pmatrix}1\\0\end{pmatrix}$,
$\begin{pmatrix}1\\1\end{pmatrix}$ in such a way that any two
consequent edges have different vectors.

\item Any $3$-dimensional simple polytope allows a characteristic
map, since there is a right colouring in $4$ colours, according to
the four colours theorem. We can assign the vectors
$\begin{pmatrix}1\\0\\0\end{pmatrix}$,
$\begin{pmatrix}0\\1\\0\end{pmatrix}$,
$\begin{pmatrix}0\\0\\1\end{pmatrix}$,
$\begin{pmatrix}1\\1\\1\end{pmatrix}$ to the facets coloured in the
colours $1$, $2$, $3$, and $4$ respectively.

\item $s(P)\geqslant m-\gamma(P)$, where $\gamma(P)$ is a {\itshape
chromatic number} of $P$ (I.Izmestiev, 2001, \cite{IZ}).

\item $s(P)\leqslant m-\lceil\log_2(\gamma(P)+1)\rceil$ (A.Aizenberg, 2009,
\cite{A}).

\item In fact, there are two dual combinatorial interpretations of
the $s$-number.
\begin{prop*}[See \cite{BP}]
\begin{enumerate}
\item $s(P)$ is the maximal positive integer such that there
exists the matrix $M$ of size $m\times s$ with integer entries
satisfying the property: each row $\ib{m}_i$ of $M$ corresponds to
the facet $F_i$ and if $F_{i_1},\dots,F_{i_n}$ intersect in a
vertex, then the columns of the matrix
$M\setminus\{\ib{m}_{i_1},\dots,\ib{m}_{i_n}\}$ form a part of some
basis of $\mathbb Z^{m-n}$.

\item $s(P)$ is the maximal positive integer such that there
exists a mapping $\mathfrak F\to\mathbb Z^{m-s}$ satisfying the
property: if $F_{i_1},\dots,F_{i_n}$ intersect in a vertex, then the
corresponding vectors form a part of some basis of $\mathbb
Z^{m-s}$.
\end{enumerate}
\end{prop*}

These two combinatorial interpretations allow us to define $s(K)$
for any simplicial complex $K$ -- we should substitute vertices of
$K$ for facets of $P$ and  simplices of maximal dimension for
vertices of $P$. Then $s(P)=s(\partial P^*)$.

\item If $s(P)=m-n$, then it is known that $b_{2i}=\rk
H_{2i}(M^{2n})=h_i(P)$ for any quasitoric manifold $M^{2n}$.

Similarly, we can consider the space $\mathcal Z_P/H$ for any toric
subgroup $H\cong T^r$ that acts freely on $\mathcal Z_P$. If $T$ is
the $(m-r)\times m$-matrix with integral entries that arises in the
second combinatorial interpretation of $s(P)$, then the cohomology
ring $\C^*(\mathcal Z_P/H)$ can be described in the following way.

\begin{defin}
The {\itshape Stanley-Reisner ring} (or the {\itshape face ring}) of
a simple polytope $P$ with $m$ facets $\{F_1,\dots, F_m\}$ is
defined as
$$
\mathtt k[P]=\mathtt k[v_1,\dots,v_m]/I,
$$
where $\mathtt k$ is a ring, each variable $v_i$ corresponds to the
facet $F_i$, and
$$
I=(\{v_{i_1},\dots,v_{i_k}:F_{i_1}\cap\dots\cap
F_{i_k}=\varnothing\}).
$$
\end{defin}

Then we have:
\begin{prop*}[See \cite{BP}]
\begin{enumerate}
\item There is an isomorphism of the algebras
$$ \C^*(\mathcal Z_P/H,\mathbb Q)\cong\Tor_{\mathbb
Q[t_1,\dots,t_{m-r}]}(\mathbb Q[P],\mathbb Q),
$$
where a structure of $\mathbb Q[t_1,\dots,t_{m-r}]$-module in the
Stanley-Reisner ring $\mathbb Q[P]$ is given by the mapping
$$
\mathbb Q[t_1,\dots,t_{m-r}]\to\mathbb Q[v_1,\dots,v_m],\quad t_i\to
t_{i1}v_1+\dots+t_{im}v_m.
$$

\item There is an isomorphism of the algebras:
$$
\C^*(\mathcal Z_P/H,\mathbb Q)\cong
\C[\Lambda[u_1,\dots,u_{m-r}]\otimes\mathbb Q[P],d],
$$
$$
du_i=t_{i1}v_1+\dots+t_{im}v_m,\quad dv_j=0.
$$
\end{enumerate}
\end{prop*}

In some sense we can consider dimensions of the cohomology groups
$\C^i(\mathcal Z_P)$ as analogues of $h$-numbers depending on the
subgroup $H$.

Let us continue the list of the properties of $P$:

\item Let us consider a simplicial complex $U_k$ (See \cite{DJ}).
Vertices of $U_k$ correspond to vectors in $\mathbb Z^k$. The set of
vertices $v_1,\dots,v_l$ form a simplex if and only if the
corresponding vectors form a part of some basis of $\mathbb Z^k$.
Then, using the second combinatorial description of the $s$-number,
we see that $s(K)$ is the maximal integer $s$ such that there exists
a non-degenerate simplicial map $K\to U_{m-s}$. (See [A]).

\item This property allows us to put $s(P)$ into the series of
invariants of the following form: given the series of simplicial
complexes $L=(L_1,L_2,\dots,L_k,\dots)$ with non-degenerate maps
$L_i\to L_{i+1}$ let us define the $L$-invariant of $K$ as the
minimal number $k$ such that there exists a non-degenerate map $K\to
L_k$.

If we take $L_k=\Delta^{k-1}$, then $L(K)=\gamma(K)$. For the
$k$-dimensional skeletons of an infinite-dimensional simplex
$L_k=\Delta^{\infty}_k$ we have $L(K)=\dim K$ (see \cite{A}).

In fact, in the case of such invariants we can consider the classes
of equivalence $K_1\sim K_2$ if and only if there exist
non-degenerate simplicial  maps $f_1:K_1\to K_2$ and $f_2:K_2\to
K_1$. Then we have the next property:

\item The $1$-dimensional skeleton of $U_k$ is equivalent to the full
graph $K_{2^k-1}$. As a consequence, for a graph $\Gamma$ we have
$s(\Gamma)= m-\lceil\log_2(\gamma(\Gamma)+1)\rceil$ (see \cite{A}).

\item In the dimension $3$ every matrix with all entries $0$ and $1$,
that has the determinant $1$ over $\mathbb Z_2$, has the determinant
$\pm 1$ over $\mathbb Z$. So in the case of  matrices $m\times 3$ we
can substitute $\mathbb Z_2$ for $\mathbb Z$. The same observation
can be used to prove  that the $2$-dimensional skeleton of $U_k$ is
equivalent to the $2$-dimensional skeleton of $U_k(2)$, where
$U_k(2)$ is a simplicial complex with vertices corresponding to the
vectors in $\mathbb Z_2^k$ and simplices corresponding to the sets
of vectors that are linearly independent over $\mathbb Z_2$.

\item There exists a non-degenerate mapping $U_k*U_l\to U_{k+l}$.

If $\sigma^{l-1}$ is a $(l-1)$-dimensional simplex in $U_k$, then
$\link_{U_k}\sigma^{l-1}\simeq U_{k-l}$ (see \cite{A}).

\item $U_k$ is $(k-2)$-connected \cite{DJ}.

\item  Let $\Delta^{m-1}_{n-1}$ be a
$(n-1)$-dimensional skeleton of a $(m-1)$-dimensional simplex. Then
\cite{A}
\begin{gather*}
m-s(\Delta^{m-1}_{n-1})\geqslant\lceil\log_2(m-n+3)\rceil+n-2;\\
m-s(\Delta^{m-1}_2)=\lceil\log_2 m \rceil+1;\\
s(\Delta^{m-1}_{m-3})=1.
\end{gather*}

\item If we substitute $\mathbb Z_2$ for $\mathbb Z$, then we can define
    $s_{\mathbb R}(K)$ in the similar way as $s(K)$ according to the
    combinatorial definitions. Then, clearly, $s(K)\leqslant s_{\mathbb
    R}(K)$. Y.~Fukukawa and M.~Masuda \cite{FM} investigated the case of
    $K=\Delta^{m-1}_{n-1}=\Delta^{m-1}_{m-1-p}$ and proved the following
    facts:

Let us denote $s(m,p)=s(\Delta^{m-1}_{m-1-p})$ and
$\sr(m,p)=\sr(\Delta^{m-1}_{m-1-p})$. Then $s(m,0)=0$, and for
$p\geqslant 1$ we have:
\begin{itemize}
\item[\textbf{I.}]
\begin{itemize}
\item[a)] $1\leqslant s(m,p)\leqslant p\;$ and $s_{\mathbb R}(m,p)=p$ if and only if
$p=1,m-1,m$.
\item[b)] $s(m,p)$ increases as $p$ increases and decreases as $m$
increases.
\item[c)] If $m-p$ is even, then $s_{\mathbb R}(m+1,p)=s_{\mathbb R}(m,p)$.
\item[d)] $s_{\mathbb R}(m+1,m-2)=s_{\mathbb R}(m,m-2)=[m-\log_2(m+1)]$
for $m\geqslant 3$.
\end{itemize}
\item[\textbf{II.}]
\begin{itemize}
\item[a)] $\sr(m,p)=1$ if and only if $m\geqslant 3p-2$ and there is a
non-negative integer $m_k(b)$ associated to $k\geqslant 2$ and
$b\geqslant 0$ such that
$$
\sr(m,p)=k\mbox{ if and only if }m_{k+1}(p-1)<m\leqslant m_k(p-1),
$$
and $m_k(b)$ is the maximum integer that the linear function
$\sum\limits_{v\in(\mathbb Z_2)^k\setminus\{0\}}a_v$ takes on
lattice points $(a_v)$ in $\mathbb R^{2^k-1}$ satisfying these
$(2^k-1)$ inequalities:
$$
\sum\limits_{(u,v)=0}a_v\leqslant b\mbox{ for each }u\in(\mathbb
Z_2)^k\setminus\{0\}
$$
and $a_v\geqslant 0$ for every $v$, where $\mathbb Z_2=\{0,1\}$ and
$(\;,\;)$ denotes the standard scalar product on $(\mathbb Z_2)^k$.
\item[b)] Let $b=(2^{k-1}-1)Q+R$ with non-negative integers $Q,R$ with
$0\leqslant R\leqslant 2^{k-1}-2$. Let $2^{k-1}-2^{k-1-l}\leqslant
R<2^{k-1}-2^{k-1-(l+1)}$ for some $0\leqslant l\leqslant k-2$. Then
$$
(2^k-1)Q+R+2^{k-1}-2^{k-1-l}\leqslant m_k(b)\leqslant (2^k-1)Q+2R,
$$
and the lower bound is attained if and only if
$R-(2^{k-1}-2^{k-1-l})\leqslant k-l-2$ and the upper bound is
attained if and only if $R=2^{k-1}-2^{k-1-l}$.
\end{itemize}
\end{itemize}
They also conjectured that $m_k((2^{k-1}-1)Q+R)=(2^k-1)Q+m_k(R)$,
supplied the conjecture with many computations and proved that
$m_k(2^{k-1}-1+b)=2^k-1+m_k(b)$ for
$b\geqslant(2^{k-1}-1)(2^{k-2}-1)$.

Let us mention that the second part of \textbf{I.}d) follows from
10., and part 9. of the main theorem follows from \textbf{II}.b) for
$k=2$ and $k=3$.

\item Every {\itshape nestohedron} can be realized as a Delzant
polytope (A.Zelevinsky, 2006, \cite{Z}). In particular, $s(P_B)=m-n$
for every nestohedron $P_B$.
\end{enumerate}
\section{Main result.}

\subsection{Construction $P\to P_{k_1,\dots,\,k_m}$}
\textbf{I.} Let $P$ be a simple polytope

$$
P=\{\ib{x}\in\mathbb R^n, A_{p}\ib{x}+\ib{b}_{p}\geqslant
0,A_p\in\mathbb R^{m\times n}, \ib{b}_p\in\mathbb R^m\}.
$$

Then the mapping $$\ib{x}\to A_{p}\ib{x}+\ib{b}_{p}$$ defines an
embedding of $P$ into $\mathbb R^m_{+}=\{\ib{y}\in\mathbb R^m:
y_i\geqslant 0,\,i=1,\dots,m\}$. We can find a matrix $C=\{C_{ij},
i=1,\dots,m-n;\,j=1,\dots,m\}$ of rank $m-n$ such that $C_pA_p=0$.
Then
$$
P=\{\ib{y}\in\mathbb R^m_{+}:C_p\ib{y}=C_p\ib{b}_p\}.
$$

Let us define a polytope $P_{k_1,\dots,\,k_m},\, k_i\geqslant 1$ as
(\cite{BBCG,BBCG2}, see also \cite{GLM}, \cite{U})
\begin{multline*}
P_{k_1,\dots,\,k_m}=\{\ib{y}=(y^1_1,\dots,y^{k_1}_1,\dots,y^1_m,\dots,y^{k_m}_m)\in\mathbb
R^{k_1+\dots+k_m}_{+}:\\
\sum\limits_{j=1}^mC_{ij}\left(\sum\limits_{l=1}^{k_j}y^{l}_j\right)=\sum\limits_{j=1}^mC_{ij}b_j,\
i=1,\dots,m-n.\}
\end{multline*}

Then
$$
P_{k_1,\dots,\,k_m}=\{\ib{y}\in\mathbb R^{k_1+\dots+k_m}_{+} :
y^1_j+\dots+y^{k_j}_j=\ib{a}_j\ib{x}+b_j, j=1,\dots,m\mbox{ for some
} \ib{x}\in \mathbb R^n\},
$$
where $\{\ib{a}_j\}$ are the row-vectors of the matrix $A_p$.

If we eliminate $y^{k_j}_j$  for $j=1,\dots, m$ and denote
$Y=(y^1_1,\dots,y^{k_1-1}_1,\dots,y^1_m,\dots,y^{k_m-1}_m)$, then
$$
P_{k_1,\dots,\,k_m}:=\{(\ib{x},Y)\in\mathbb
R^{n+(k_1-1)+\dots+(k_m-1)} :
\ib{a}_j\ib{x}+b_j\geqslant\sum\limits_{l=1}^{k_j-1}y^l_j,\quad
y^l_j\geqslant 0\}.
$$

\begin{prop}
$P_{k_1,\dots,\,k_m}$ is a simple polytope of dimension
$n+(k_1-1)+\dots+(k_m-1)$ and
$$
\mathbb Z[P_{k_1,\dots,\,k_m}]= \mathbb
Z[v^1_1,\dots,v^{k_1}_1,\dots,v^1_m,\dots,v^{k_m}_m]/J,\quad
J=(v^1_{i_1}\dots v^{k_{i_1}}_{i_1}\dots v^1_{i_l}\dots
v^{k_{i_l}}_{i_l}: v_{i_1}\dots v_{i_l}\in I).
$$
\end{prop}
\begin{proof}
Consider a vertex $v$ of $P_{k_1,\dots,\,k_m}$. If the number of
facets intersecting in $v$ is greater than
$n+(k_1-1)+\dots+(k_m-1)$, then there are at least $n+1$ indices
$\{i_1,\dots,i_{n+1}\}$ such that $y^l_{i_s}=0$ for all
$l=1,\dots,k_{i_s}-1$ and $\ib{a}_{i_s}\ib{x}+b_{i_s}=0, s=1,\dots,
n+1$ for some $\ib{x}\in \mathbb R^n$. But
$\ib{a}_j\ib{x}+b_j\geqslant\sum\limits_{l=1}^{k_j-1}y^l_j\geqslant
0$ for all $j$, so $\ib{x}\in P$. Thus we have a contradiction,
because $P$ is simple.

So each vertex $v$ of $P_{k_1,\dots,\,k_m}$ belongs to exactly
$n+(k_1-1)+\dots+(k_m-1)$ facets. In fact, to avoid the contradiction there
should be $n$ indices $\{i_1,\dots,i_n\}$ such that $y^l_{i_s}=0$ for all
$l=1,\dots,k_{i_s}-1$ and $\ib{a}_{i_s}B+b_s=0,\, s=1,\dots,n$. Here $B$ is
the corresponding vertex of $P$. All other variables
$y^l_j,j\notin\{i_1,\dots,i_n\}$ should satisfy the following condition. For
each $j$ either all variables $y^1_j,\dots,y^{k_j-1}_j$ are equal to zero or
only one of them is nonzero and it is equal to $\ib{a}_jB+b_j$. It is easy to
see that in each case we have a vertex of $P_{k_1,\dots,\,k_m}$.

Let us denote the facet $y^l_j=0$ by $F^l_j$ and the facet
$\ib{a}_j\ib{x}+b_j=\sum\limits_{l=1}^{k_j-1}y^l_j$ by $F^{k_j}_j$.
Now consider some collection $\mathcal G$ of facets of
$P_{k_1,\dots,\,k_m}$. Let $\{i_1,\dots,i_l\}$ be the set of indices
such that $F^1_{i_s},\dots,F^{k_{i_s}}_{i_s}\in \mathcal G$ for each
$i_s$. If $F_{i_1}\cap\dots\cap F_{i_l}\ne\varnothing$, then we can
find a vertex $B=F_{i_1}\cap\dots\cap F_{i_l}\cap
F_{i_{l+1}}\cap\dots\cap F_{i_n}$. For each index
$j\notin\{i_1,\dots,i_n\}$ there are at most $k_j-1$ facets in
$\mathcal G$, so we can construct some vertex $v$ of
$P_{k_1,\dots,\,k_m}$ that lies in all facets in $\mathcal G$.

On the other hand, if $F_{i_1}\cap\dots\cap F_{i_l}=\varnothing$,
then there are no vertices of $P_{k_1,\dots,\,k_m}$ that belong to
all facets in $\mathcal G$.

So the collection $\mathcal G$ defines a face if and only if the
corresponding monomial in $\mathbb Z[P_{k_1,\dots,\,k_m}]$ isn't
divided by a monomial of the form $v^1_{i_1}\dots
v^{k_{i_1}}_{i_1}\dots v^1_{i_l}\dots v^{k_{i_l}}_{i_l}$, where
$v_{i_1}\dots v_{i_l}\in I$.
\end{proof}

\textbf{II.} Let $\lambda=\langle \lambda, \ib{x}\rangle \in(\mathbb
R^n)^*$ be a generic linear function, that is it takes different
values on vertices connected by an edge. This function induces the
orientation of edges from the smaller vertex to the greater. Then
each vertex $v=F_{i_1}\cap\dots\cap F_{i_n}$ corresponds to the
monomial $m_v=x_{i_1}\dots x_{i_n}\in\mathbb
Z[\alpha_1,t_1,\dots,\alpha_m,t_m]$, where
$$
x_{i_k}=\begin{cases}
\alpha_{i_k},& \mbox{ if the edge }
e=F_{i_1}\cap\dots\widehat{F_{i_k}}\cap\dots\cap F_{i_n} \mbox{ is
ingoing, and}\\
t_{i_k}, &\mbox{ if the edge } e \mbox{ is outgoing.}
\end{cases}
$$

Then we can consider the $h$-polynomial
$$
h_{\lambda}(P)(\alpha_1,t_1,\dots,\alpha_m,t_m)=\sum\limits_{v}m_v.
$$
It is known that
$h_{\lambda}(P)(\alpha,t,\alpha,t,\dots,\alpha,t)=H(\alpha,t)=\sum\limits_{i=0}^nh_i\alpha^{n-i}t^i$
is a usual $H$-polynomial.

Each vertex $B=F_{i_1}\cap\dots\cap F_{i_n}$ of $P$ corresponds to
$k_1\dots\widehat{k_{i_1}}\dots\widehat{k_{i_n}}\dots k_m$ vertices
of  $P_{k_1,\dots,\,k_m}$, which are obtained by choosing for each
$j\notin\{i_1,\dots,i_n\}$ one facet $F^l_j$ that doesn't contain
the vertex.

Let us take the linear function $\Lambda=(\lambda,\ib{c})\in(\mathbb
R^{n+(k_1-1)+\dots+(k_m-1)})^*$, where
$$
\ib{c}=(c^1_1,\dots,c^{k_1-1}_1,\dots,c^1_m,\dots,c^{k_m-1}_m),\quad
c^1_j<c^2_j<\dots<c^{k_j-1}_j<0,\quad \mbox{and }|c^l_j|\ll\lambda.
$$
The vertex $v=(B,Y)$ is connected by edges with
$(k_1-1)+(k_2-1)+\dots+\widehat{(k_{i_1}-1)}+\dots+\widehat{(k_{i_n}-1)}+\dots+
(k_m-1)$ vertices $v'$ of the same form, that are obtained from $v$
by changing $\{y^l_j\}$ for some $j$. In each case
$\langle\Lambda,v\rangle<\langle\Lambda,v'\rangle$ if and only if we
substitute $F^{l'}_j$ for $F^l_j$ with $l<l'$.

The rest $n+(k_{i_1}-1)+\dots+(k_{i_n}-1)=k_1+\dots+k_n$ adjacent
vertices have the form $(B',Y')$, where
$B'=F_{i_1}\cap\dots\cap\widehat{F_{i_s}}\cap\dots\cap F_{i_n}\cap
F_t$ is adjacent to $B$ in $P$,
$y'^1_{i_s}=\dots=y'^{k_{i_s}-1}_{i_s}=0$ or
$y'^1_{i_s}=\dots=\widehat{y'^l_{i_s}}=\dots=y'^{k_{i_s}-1}_{i_s}=0$
and $y'^l_{i_s}=\ib{a}_{i_s}B'+b_{i_s}$ for some $l$, and
$$
y'^l_j=\begin{cases}
0,&j\in\{i_1,\dots,\widehat{i_s},\dots,i_n,t\},\\
0,&j\notin\{i_1,\dots,i_n,t\}\mbox{ and } y^l_j=0,\\
\ib{a}_jB'+b_j,&j\notin\{i_1,\dots,i_n,t\}\mbox{ and }
y^l_j=\ib{a}_jB+b_j.
\end{cases}
$$

In this case the edge connecting $(B,Y)$ and $(B',Y')$ is ingoing if
an only if the corresponding edge connecting $B$ and $B'$ in $P$ is
ingoing.

Then the sum of monomials $m_{(B,Y)}$ over all vertices $(B,Y)$ with
fixed $B$ is equal to
$$
\sum\limits_{(B,Y)}m_{(B,Y)}=x^1_{i_1}\dots
x^{k_{i_1}}_{i_1}\dots x^1_{i_n}\dots
x^{k_{i_n}}_{i_n}\prod\limits_{j\notin\{i_1,\dots,i_n\}}(\alpha^1_j\dots\alpha^{k_j-1}_j+\alpha^1_j\dots\alpha^{k_j-2}_jt^{k_j}_j+\dots+\alpha^1_jt^3_j\dots
t^{k_j}_j+t^2_j\dots t^{k_j}_j),
$$
where $x^l_j=\alpha^l_j$ if $x_j=\alpha_j$ and $x^l_j=t^l_j$ if
$x_j=t_j$.

If we denote
$\gamma_j=\alpha^1_j\dots\alpha^{k_j-1}_j+\alpha^1_j\dots\alpha^{k_j-2}_jt^{k_j}_j+\dots+\alpha^1_jt^3_j\dots
t^{k_j}_j+\dots+t^2_j\dots t^{k_j}_j$, then we have:
\begin{prop}
$$
h_{\Lambda}(P_{k_1,\dots,\,k_m})=\gamma_1\dots\gamma_m
h_{\lambda}(P)(\frac{\alpha^1_1\dots\alpha^{k_1}_1}{\gamma_1},\frac{t^1_1\dots
t^{k_1}_1}{\gamma_1},\dots,\frac{\alpha^1_m\dots\alpha^{k_m}_m}{\gamma_m},\frac{t^1_m\dots
t^{k_m}_m}{\gamma_m}).
$$
\end{prop}

\textbf{III.} It should be mentioned that the previous construction
can be easily handled using {\itshape Gale transforms}.
\begin{defin}
Let $V=(\ib{v}_1,\dots,\ib{v}_m)$ be a point configuration in $\mathbb R^n$
such that $\aff(V)=\mathbb R^n$. Let us find a basis
$(\ib{u}_1,\dots,\ib{u}_{m-n-1})$ in the space of affine dependences
$$
\{\ib{c}\in\mathbb R^m:
c_1\bigl(\begin{smallmatrix}\ib{v}_1\\1\end{smallmatrix}\bigr)+\dots+c_m\bigl(\begin{smallmatrix}\ib{v}_m\\1\end{smallmatrix}\bigr)=0\}.
$$
Consider an $m\times(m-n-1)$-matrix $U$ formed by the column vectors
$\ib{u}_1,\dots,\ib{u}_{m-n-1}$. Then the row vectors
$\overline{\ib{v}}_1,\dots,\overline{\ib{v}}_m\in\mathbb R^{m-n-1}$ of $U$
form a {\itshape dual configuration} or a {\itshape Gale transform}. The
{\itshape Gale diagram} of $V$ is the set of points $\{\ib{x}_i\}$ on
$S^{m-n-2}\cup\{\mathbf 0\}$:
$$
\ib{x}_i=\begin{cases} \mathbf 0,& \mbox{ if
}\overline{\ib{v}}_i=\mathbf 0,\\
\frac{\overline{\ib{v}}_i}{|\overline{\ib{v}}_i|},&\mbox{ if }
\overline{\ib{v}}_i\ne \mathbf 0.
\end{cases}
$$
\end{defin}

It is known (See, for example, \cite{GB}) that if $\ib{v}_1,\dots,\ib{v}_m$
are vertices of a polytope $P$, then $\ib{v}_{i_1},\dots,\ib{v}_{i_k}$ form a
face of $P$ if and only if
$0\in\relint\{\overline{\ib{v}}_1,\dots,\widehat{\overline{\ib{v}}_{i_1}},\dots,\widehat{\overline{\ib{v}}_{i_k}},\dots,\overline{\ib{v}}_m\}$.
The same is true for Gale diagrams.

Let $P\in\mathbb R^n$ be a simple polytope with $m$ facets and $0\in
P$, $P^*$ -- it's dual polytope and
$\overline{\ib{v}_1},\dots,\overline{\ib{v}_m}$ -- a configuration
dual to the set $\{v_1,\dots,v_m\}$ of vertices of $P^*$. Then if we
take each point $\overline{\ib{v}_j}\in\mathbb R^{m-n-1}$ of the
dual configuration $k_j$ times and come back to $\mathbb R^n$, then
we obtain a simplicial polytope $(\widehat{P})^*$, and a dual
polytope $\widehat{P}$ is combinatorially equivalent to
$P_{k_1,\dots,\,k_m}$.

This observation is a bridge to simple polytopes with $n+3$ facets.

\subsection{Simple polytopes with $n+3$ facets.}

\textbf{I.} It was invented by M.Perles (see the classical book by
B.Grunbaum \cite{GB}) that any $n$-polytope with $n+3$ vertices  is
combinatorially equivalent to the polytope
$P_{a_0,a_1,\dots,\,a_{2k}}$ that has the Gale diagram of the
following type. Let $M_{2k}$ be a regular $2k$-gon inscribed in the
unit circle with the center $\mathbf 0$ and let $v_1,\dots,v_{2k}$
be the set of it's vertices in the order they lie on the circle.
Then the Gale diagram of $P_{a_0,\dots,\, a_{2k}}$ consist of $a_0$
times $\mathbf 0$, $a_i$ times $v_i$ for $i=1,\dots, 2k$,
$a_0+a_1+\dots+a_{2k}=n+3,\, a_i\geqslant 0$,  with the property
that there are no two consequent vertices with $a_i=0$.

In the case of simplicial polytopes $a_0$ should be equal to $0$ and
there should be no diameters of the circle with both endpoints in
the configuration. So we see that each simplicial polytope $P^n$
with $n+3$ vertices can be described in terms of regular
$(2k-1)$-gon $M_{2k-1}$ and a surjective map from the set of
vertices $V_1,\dots,V_{n+3}$ of $P^n$ to $\{v_1,\dots,v_{2k-1}\}$
with the property that the set of vertices
$\{V_1,\dots,\widehat{V_{i_1}},\dots,\widehat{V_{i_2}},\dots,\widehat{V_{i_3}},\dots,V_{n+3}\}$
form an $(n-1)$-face if and only if $\mathbf 0\in\relint
\{V_{i_1},V_{i_2},V_{i_3}\}$ (here we denote $V_{i_k}$ and it's
image by the same letter). This result is equally valid for simple
polytopes if we substitute ``facets'' for ``vertices''.

Now let us consider the ``simplest'' polytopes with $n+3$ facets,
that is,  polytopes $P_k$ corresponding to regular $(2k-1)$-polygons
with multiplicity of each vertex $1$. $5$-gon corresponds to usual
$5$-gon. It's face ring has the form $\mathbb Z[P_5]=\mathbb
Z[v_1,v_2,v_3,v_4,v_5]/(v_1v_2,v_2v_3,v_3v_4,v_4v_5,v_5v_1)$

Let us find a face ring for any $k$. The set of facets corresponding
to $\{v_{i_1},\dots,v_{i_r}\}$ do not intersect if and only of it's
complement $\{v_j: j\notin\{i_1,\dots, i_r\}\}$ doesn't contain a
triangle with $\mathbf 0$ in the relative interior. We claim that it
can happen if and only if the set
$\{v_{i_1},v_{i_2},\dots,v_{i_r}\}$ contains $k-1$ successive
vertices $v_i,v_{i+1},\dots, v_{i+k-2}$, where here and below we
consider the indices modulo $(2k-1)$.

To prove this fact let us consider the longest sequence of vertices
$v_i,v_{i+1},\dots,v_{i+l-1}$ in $\{v_{i_1},\dots,v_{i_r}\}$. Assume
that $l<k-1$. Consider two vertices $v_{i-1}$ and $v_{i+l}$. Then
for any $j:(i-1)+k\leqslant j\leqslant i+l+(k-1)$ we have $\mathbf
0\in\relint \conv\{v_{i-1},v_{i+l},v_j\}$, so
$j\in\{i_1,\dots,i_r\}$. Thus we have a segment of length $l+1$ and
a contradiction.

So we have:
\begin{prop}
$$
\mathbb Z[P_k]=\mathbb Z[v_1,\dots,v_{2k-1}]/(v_iv_{i+1}\dots
v_{i+k-2}),
$$
where the indices are taken modulo $(2k-1)$.
\end{prop}

Observing that the polytope corresponding to regular $(2k-1)$-gon
with multiplicities $a_1,\dots, a_{2k-1}$ is combinatorially
equivalent to $(P_k)_{a_1,\dots,a_{2k-1}}$ we can formulate a
corollary:

\begin{prop}
The Stanley-Reisner ring of a simple polytope with $n+3$ facets
corresponding to the $(2k-1)$-gon with multiplicities
$a_1,\dots,a_{2k-1}$ is equal to
$$
\mathbb
Z[v^1_1,\dots,v^{a_1}_1,\dots,v^1_{2k-1},\dots,v^{a_{2k-1}}_{2k-1}]/(v^1_i\dots
v^{a_i}_iv^1_{i+1}\dots v^{a_{i+1}}_{i+1}\dots v^1_{i+k-2}\dots
v^{a_{i+k-2}}_{i+k-2}),
$$
where we use a notation $v^l_{i+(2k-1)}=v^l_i$, and $1\leqslant
i\leqslant 2k-1$.
\end{prop}

At last we can name the polytope $P_k$. In fact it is
combinatorially equivalent to the polytope $(C^{2k-4}(2k-1))^*$ dual
to the {\itshape cyclic polytope} $C^{2k-4}(2k-1)$.

\begin{defin}
A {\itshape Cyclic polytope} $C^n(m)$ is a combinatorial simplicial
polytope that can be defined in the following way:
$C^n(t_1,\dots,t_m)$ is a convex hull of the points
$\{\ib{t}_i=(t_i,t_i^2,\dots,t_i^n)\in\mathbb R^n, t_1<\dots<t_m\}$
on the {\itshape moment curve} $(t,t^2,\dots,t^n)$. It is known
(see, for example, \cite{BP} or \cite{GB}) that each polytope
$C^n(t_1,\dots,t_m)$ is a simplicial polytope and it's face lattice
is defined by the following ``Gale's evenness condition'': the set
$\omega=\{i_1,\dots,i_n\}$ defines the facet
$\conv\{\ib{t}_{i_1},\dots,\ib{t}_{i_n}\}$ if and only if any two
points $\ib{t}_{j_1},\ib{t}_{j_2}: j_1,j_2\notin\omega$ are
separated on the moment curve by an even number of points of
$\omega$.

Thus all the polytopes $C^n(t_1,\dots,t_m)$ are combinatorially
equivalent and their combinatorial type is denoted by $C^n(m)$.
\end{defin}

\begin{prop}
$P_k$ is combinatorially equivalent to $(C^{2k-4}(2k-1))^*$.
\end{prop}
\begin{proof}
Denote by $F_i$ the facet of $(C^{2k-4}(2k-1))^*$ corresponding to
the vertex $\ib{t}_i$ of $C^{2k-4}(2k-1)$.

At first let us note that the cyclic polytope $C^n(m)$ of even
dimension has a rotational symmetry, that is we can imagine that the
points $\ib{t}_1,\dots,\ib{t}_m$ are situated on the circle (so the
point $\ib{t}_1$ is next to $\ib{t}_m$) and the set
$\omega=\{i_1,\dots,i_n\}$ defines the facet
$\conv\{\ib{t}_{i_1},\dots,\ib{t}_{i_n}\}$ if and only if any to
points $\ib{t}_{j_1},\ib{t}_{j_2}: j_1,j_2\notin\omega$ are
separated on the circle by an even number of points of $\omega$.

We use the term {\itshape segment} to denote the sequence of
adjacent points on the circle and the {\itshape length} of a segment
is just the number of points in it.

Indeed, any set $\omega$ with the above property defines a facet.
Problems occur when we move in the opposite direction: any vertex of
the cyclic polytope is defined by such a set $\omega$. Let us
consider the set $\omega$ that defines a vertex. Then each pair
$\ib{t}_{j_1},\ib{t}_{j_2}: j_1,j_2\notin \omega$ are separated on
the segment $[1,\dots,m]$ by even number of points of $\omega$. But
if we take the other arc of the circle, then we see that all the
points of $\omega$ lie on segments $[i,i+1,\dots,i+l]$ of even
length except for the the points $\{1,\dots,i\}$ and
$\{m-j+1,\dots,m\}\in\{i_1,\dots,i_n\}$ at the ends of the segment
$[1,\dots,m]$. If $n$ is even, then $i+j$ should be even and we have
the proof, and if $n$ is odd, then our statement is not true.

Let us define the correspondence $\varphi:F_i\to V_{ki}$. The
greatest common divisor of $k$ and $2k-1$ is equal to $1$, so our
correspondence $\varphi$ is a bijection.

On the other hand, let $n=2k-4$ and
$\Omega=\{V_1,\dots,\widehat{V_{i_1}},\dots,\widehat{V_{i_2}},\dots,
\widehat{V_{i_3}},\dots, V_{n+3}\}$ be a set of facets that doesn't
intersect in a vertex of $P_k$. Then this set contains some segment
$V_i,V_{i+1},\dots, V_{i+k-2}$ of facets of $P_k$. If $V_i=V_{kj}$,
then $V_{i+1}=V_{kj+(2k-1)+1}=V_{k(j+2)}$. The facets
$V_{i_1},V_{i_2},V_{i_3}$ lie in the complement
$\{V_{i+k-1},V_{i+k},\dots,V_{i-1}\}$, which consists of $k$ facets
with the property that if $V_j$ corresponds to $F_l$, then $V_{j+1}$
-- to $F_{l+2}$. So for two adjacent facets $V_{i_s}$ and $V_{i_t}$
their preimages $\varphi^{-1}(V_{i_s})$ and $\varphi^{-1}(V_{i_t})$
are separated by an odd number of facets of $\varphi^{-1}(\Omega)$,
and $\varphi^{-1}(\Omega)$ is not a vertex of $(C^{2k-4}(2k-1))^*$.

Now let us consider the set
$\omega=\{1,\dots,2k-1\}\setminus\{i_1,i_2,i_3\}$:
$F_1\cap\dots\widehat{F_{i_1}}\dots\widehat{F_{i_2}}\dots\widehat{F_{i_3}}\dots\cap
F_{2k-1}$ is not a vertex of $(C^n(n+3))^*$. According to the
``Gale's evenness condition'' at least two of the facets
$\{F_{i_1},F_{i_2},F_{i_3}\}$ should be separated by an odd number
of facets. Let it be $F_{i_1}$ and $F_{i_2}$. It means that one of
the arcs between $F_{i_1}$ and $F_{i_2}$ contains an odd number of
points and the other is divided by $F_{i_3}$ into two arcs : one
consisting of an odd number of points and the other -- of an even.
Without loss of generality let the segments  $(F_{i_1},F_{i_2})$ and
$(F_{i_2}, F_{i_3})$ contain an odd number of points. Then
$\varphi(F_{i_1}), \varphi(F_{i_2})$ and $\varphi(F_{i_3})$ lie on
the segment of length at most $k$ and thus the set of facets
$\Omega=\{\varphi(V_{j}):j\in\omega\}$ contains the segment of
length $k-1$ and doesn't intersect in a vertex of $P_k$.

We have built an explicit combinatorial equivalence between
$(C^{2k-4}(2k-1))^*$ and $P_k$, but the fact that they are
equivalent can be seen much more easy: $P_k$ is the only neighbourly
polytope among the polytopes of dimension $2k-4$ with $2k-1$ facets.
It is easy to see: if the polytope $P$ corresponds to the
$(2l-1)$-gon with $l<k$, then there exist $k-2$ facets that do not
intersect. If it is not true, then any segment $[i,\dots,i+l-2]$ of
length $l-1$ contains $s_i\geqslant k-1$ facets of $P$. Let us
consider the sum $\Sigma=\sum\limits_{i=1}^{2l-1} s_i$ over all such
segments. Each facet of $P$ should lie in $l-1$ segments, so
$\Sigma=(2k-1)(l-1)=\sum\limits_{i=1}^{2l-1}s_i\geqslant(2l-1)(k-1)$.
It is equivalent to the inequality $l\geqslant k$.
\end{proof}

We can formulate the corollary:
\begin{cor}
Any simple polytope $P^n$ with $n+3$ facets is combinatorially
equivalent to \linebreak
$P_{a_1,\dots,\,a_{2k-1}}=(C^{2k-4}(2k-1))^*_{a_1,\dots,\,a_{2k-1}}$
for some $k\geqslant 2$ and $\{a_i\geqslant
1:\sum\limits_{i=1}^{2k-1}a_i=n+3\}$.
\end{cor}

\textbf{II.} In fact, there is another approach to simple polytopes
with $n+3$ facets (compare with \cite{GB}). It is known that any
simple polytope with $n+2$ facets is projectively equivalent to the
product of two simplices $\Delta^i\times\Delta^j, j=n-i$. We can
realize it as
$$
\{(\ib{x},\ib{y})\in\mathbb
R^{i+1}(x_0,\dots,x_i)\times\mathbb
R^{j+1}(y_0,\dots,y_j):x_k\geqslant 0,y_l\geqslant
0,x_0+\dots+x_i=1, y_0+\dots+y_j=1\}.
$$

Then any simple polytope $P$ with $n+3$ facets is projectively
equivalent to the section of some $\Delta^i\times\Delta^j$ by a
halfspace
$$
a_0x_0+\dots+a_ix_i+b_0y_0+\dots+b_jy_j\leqslant c,\,
a_0<a_1<\dots<a_i,\, b_0<b_1<\dots<b_j,
$$
and since the polytope is simple up to an isotopy of a polytope we
can take $c=\pm 1$. In the case of $c=1$ the set of vertices of $P$
consists of three types:
\begin{itemize}
\item vertices $(\ib{e}_k,\ib{f}_l)$, where $\ib{e}_k$ is the $k$-th basis
vector in $\mathbb R^{i+1}$, $\ib{f}_l$ is the $l$-th basis vector
in $\mathbb R^{j+1}$, and $a_k+b_l<1$;
\item intersections of the additional hyperplane with the edges
$\ib{e}_k\times\conv\{\ib{f}_p,\ib{f}_q\},\, p<q$. In this case
$a_k+t_1b_p+t_2b_q=1, t_1,t_2>0, t_1+t_2=1$. This condition holds if
and only if $a_k+b_p<1<a_k+b_q$.
\item intersections with the edges $\conv\{\ib{a}_p,\ib{a}_q\}\times\ib{f}_l,\, p<q$. In this case
$a_p+b_l<1<a_q+b_l$.
\end{itemize}

Thus we can represent the polytope $P$ in terms of a $(i+1)\times
(j+1)$ table (compare with ``star diagram'' in \cite{GB}) on the
plane with horizontal lines $y=b_l$, vertical lines $x=a_k$ and a
line $l: x+y=1$.

Then the facets of the polytope $P$ correspond to the lines and the
vertices -- to
\begin{itemize}
\item the nodes $(a_k,b_l)$ beneath the line $l$. In such vertices
intersect all facets but those corresponding to $a_k$, $b_l$ and the
line $l$.
\item the pairs of nodes $\{(a_k,b_p),(a_k, b_q)\}$ in
different halfspaces with respect to $l$. In such vertices intersect
all facets but those corresponding to $a_k,b_p,b_q$.
\item the pairs of nodes $\{(a_p,b_l),(a_q,b_l)\}$ in different
halfspaces with respect to $l$. In such vertices intersect all
facets but those corresponding to $a_p,a_q,b_l$.
\end{itemize}

Since the case of $c=-1$ just changes the notions ``above'' and
``beneath'' the line $l$ we see that the case $c=1$ includes all
combinatorial types.

Let us mention that both tables and polygons include also the case
of simple $n$-polytopes with $n+1$ and $n+2$ facets.

\textbf{III.} Now let us describe the correspondence between tables
and polygons. In fact vertices of the $(2k-1)$-gon correspond to the
collections of successive horizontal or vertical lines that are
situated similarly in the table. More precise let us forget about
the distances and think that all the vertical (respectively
horizontal) lines are equidistant. Then the line $l$ transforms to
the broken line. We can consider $l$ as a graph of a piecewise
linear map. A combinatorial polytope is defined by the segments on
vertical and horizontal lines that intersect a broken line. Let us
put to the same class all the facets corresponding to the lines
above the broken line $l$ and $l$ itself. Each horizontal line is
divided by vertical lines into $i$ segments. Then let us put to the
same class all horizontal lines that are intersected by $l$ in the
same segment and let us gather vertical lines into classes according
to the same rule. Then obtained classes are exactly the vertices of
the $(2k-1)$-gon and they have the same order on the circle as in
the table.

\subsection{Flips and bistellar transformations}

There are very important operations called {\itshape flips} defined
on the set of simple polytopes. Let
$$P^n=\{\ib{x}\in\mathbb R^n: A_p\ib{x}+\ib{b}_p\geqslant\mathbf 0\}$$
be a simple polytope and let us take the hyperplane
$\ib{a}_s\ib{x}+b_s=0$, that defines a facet $F_s$. Then we can move
this hyperplane inside the polytope (for example, we can decrease
$b_i$, or change both $\ib{a}_i$ and $b_i$). Then until some moment
of time the combinatorial type of the polytope $P$ doesn't change,
but then we cut off one new vertex $v$, that was connected by an
edge with the facet $F_s=\{\ib{x}\in P^n: \ib{a}_s\ib{x}+b_s=0\}$
and obtain new combinatorial polytope $\hat P$.

During the movement the following event occurs: let the vertex $v$
be connected with exactly $i-1$ vertices on the facet $F_s$. Then
after the flip the vertex $v$ and all these $i-1$ vertices vanish
and new $n-i+1$ vertices corresponding to the rest $n-i+1$ edges in
$v$ appear.

Combinatorially this operation can be described quite easy: we have
$n+1$ facets\\ $F_{j_1}=F_s,F_{j_2},\dots, F_{j_n},F_{j_{n+1}}$ such
that for any $t\in I=\{1,\dots,i\}$ the intersection
$F_{j_1}\cap\dots\widehat{F_{j_t}}\dots\cap F_{j_{n+1}}$ is a
vertex, for any $t\in J=\{i+1,\dots,n+1\} :
F_{j_1}\cap\dots\widehat{ F_{j_1}}\dots\cap
F_{j_{n+1}}=\varnothing$, and $F_{j_1}\cap\dots\cap
F_{j_i}=\varnothing$. (Such a flip is called an $i$-{\itshape
flip}.) Then the flip exchanges two sets $I$ and $J$. In fact, if we
consider the simplicial complex $\partial(P^*)$ then this operation
corresponds to a bistellar $(i-1)$-transformation.

It is easy to calculate that the $H$-polynomial
$H(P)=h_0\alpha^n+h_1\alpha^{n-1}t+\dots+h_{n-1}\alpha
t^{n-1}+h_nt^n$ changes under an $i$-flip in the following way:
$$
H(\hat P)=H(P)+\frac{\alpha^{n+1-i}t^i-\alpha^it^{n+1-i}}{\alpha-t}.
$$

It will be convenient to use a usual $h$-polynomial

$$
h(t)=H(t,1)=H(1,t)=h_0t^n+h_1t^{n-1}+\dots+h_{n-1}t+h_n.
$$

An $h$-polynomial changes under an $i$-flip in the following way:
$$
h(\hat P)=h(P)=\frac{t^{n+1-i}-t^i}{t-1}.
$$

Let us consider what happens when we move lines in the table of
$n+3$ polytope.

Imagine that we move the line $l$ in such a way that exactly one
node $(a_p,b_q)$ corresponding to the vertex $v$ becomes upper.

Let us calculate the type of this flip. All the facets but those
corresponding to $a_p$ and $b_q$ take part in the flip. The moving
hyperplane corresponds to $l$, $a_p$, or $b_q$. The facets
complementary to $a_p,b_q,l$ intersect in $v$ and the complement to
the facets $a_p,a_s,b_q$ or $a_p,b_q,b_t$ intersect in a vertex if
and only if $s>p$ or $t>q$ respectively. Thus the flip has type
$(i-p)+(j-q)+1=i+j+1-(p+q)=n+1-(p+q)$. Using this operation we can
move the line $l$ to the left bottom position when just the vertex
$(a_0,b_0)$ is beneath it. In this case our polytope is a simplex
and $h(P)=\frac{t^{n+1}-1}{t-1}=\frac{t^{n+1-(0+0)}-t^{0+0}}{t-1}$.

In is easy to see that in terms of a polygon the corresponding flip
can be described in the following way:

\begin{prop}
The flip mentioned above transforms the $(2k-1)$-gon
$(a_1,\dots,a_p,\dots,a_{p+k},\dots,a_{2k-1})$ into the $(2k+1)$-gon
$(a_1,\dots,a_p-1,1,a_{p+1},\dots,a_{p+k-1},1,a_{p+k}-1,\dots,a_{2k-1})$.
If $a_p$ (or $a_{p+k}$) is equal to $1$, then  we should substitute
one vertex with the label $a_{p+k-1}+1$ (or $a_{p+1}+1$) for two
vertices with the labels $a_{p+k-1}$ and $1$ (or $a_{p+1}$ and $1$
respectively). The type of the flip is equal to
$a_{p+1}+\dots+a_{p+k-1}$.
\end{prop}

During the motion of $l$ we intersect all the nodes beneath $l$
except for $(0,0)$, so we have the statement:
\begin{prop}
Let $P$ be a simple polytope with the table $\{a_p,b_q;l\}$. Then
$$
h(P)=\frac{1}{t-1}\left(\sum\limits_{a_p+b_q<1}t^{n+1-(p+q)}-t^{p+q}\right)
$$
\end{prop}

Now let us calculate the $h$-polynomial of $P$ in terms of a
polygon. Let $P=P_{a_1,\dots,\,a_{2k-1}}$. Then we can build the
following table, corresponding to $P$:
\begin{center}
\begin{picture}(150,130)

\put(30,20){\line(1,0){90}} \put(30,35){\line(1,0){90}}
\put(30,50){\line(1,0){90}} \put(30,65){\line(1,0){90}}
\put(30,80){\line(1,0){90}}
\put(30,95){\line(1,0){90}}\put(30,110){\line(1,0){90}}

\put(5,20){$a_{2k-1}$}
\put(5,35){$a_{2k-2}$}\put(15,65){\dots}\put(10,92){$a_{k+1}$}\put(20,110){$0$}

\put(30,20){\line(0,1){90}} \put(45,20){\line(0,1){90}}
\put(60,20){\line(0,1){90}} \put(75,20){\line(0,1){90}}
\put(90,20){\line(0,1){90}}
\put(105,20){\line(0,1){90}}\put(120,20){\line(0,1){90}}

\put(30,10){$a_1$}\put(42,10){$a_2$}\put(72,10){\dots}\put(95,10){$a_{k-1}$}\put(120,10){$a_{k}-1$}

\put(30,103){\line(1,-1){83}}

\end{picture}
\end{center}

Here in the table a line with a label $a_p$ corresponds to $a_p$
parallel lines situated close to each other. Then
$$
h(P)=\frac{1}{t-1}\left(\sum\limits_{p,\,q}
t^{n+1-(p+q)}-t^{p+q}\right),
$$
where the sum is taken over all nodes beneath the line $l$.

One vertical line with a label $a_l$ corresponds to $a_l$ parallel
lines.  This line intersect exactly $k-l$ classes of horizontal
lines with labels $a_{2k-1},\dots,a_{k+l}$. So the sum over all
nodes in $i$-th line in the class $a_l$ is equal to
\begin{multline*}
\frac{1}{t-1}\left(\sum\limits_{j=0}^{a_{k+l}+\dots+a_{2k-1}-1}t^{n+1-(a_1+\dots+a_{l-1}+i-1+j)}-t^{a_1+
\dots+a_{l-1}+i-1+j}\right)=\\
=\frac{1}{t-1}\left(t^{n+1-(a_1+\dots+a_{l-1}+i-1)+(a_{k+l}+\dots+a_{2k-1}-1)}\frac{t^{a_{k+l}+\dots+a_{2k-1}}-1}{t-1}-t^{a_1+\dots+a_{l-1}+i-1}\frac{t^{a_{k+l}+\dots+a_{2k-1}}-1}{t-1}\right)=\\
\frac{1}{(t-1)^2}\left(t^{a_{k+l}+\dots+a_{2k-1}}-1\right)\left(t^{n+1-(a_{k+l}+\dots+a_{2k-1}+a_1+\dots+a_{l-1}-1)-(i-1)}-t^{(a_1+\dots+a_{l-1})+(i-1)}\right).
\end{multline*}

Let us take the sum over $i=1,\dots,a_l$. Using the equalities
$\sum\limits_{i=1}^{2k-1}a_i=n+3$ and
$$
\sum\limits_{i=1}^{a_l}t^{i-1}=\frac{t^{a_l}-1}{t-1},\qquad\sum\limits_{i=1}^{a_l}t^{-(i-1)}=t^{-(a_l-1)}\sum\limits_{i=0}^{a_l-1}t^i=t^{-(a_l-1)}\frac{t^{a_l}-1}{t-1}
$$
we obtain
\begin{multline*}
\frac{1}{(t-1)^3}\left(t^{a_{k+l}+\dots+a_{2k-1}}-1\right)\left(t^{a_l}-1\right)\left(t^{n+1-(a_{k+l}+\dots+a_{2k-1}+a_1+\dots+a_{l-1}-1)}t^{-(a_l-1)}-t^{a_1+\dots+a_{l-1}}\right)=\\
=\frac{1}{(t-1)^3}\left(t^{a_{k+l}+\dots+a_{2k-1}}-1\right)(t^{a_l}-1)\left(t^{a_{l+1}+\dots+a_{l+k-1}}-t^{a_1+\dots+a_{l-1}}\right)=\frac{1}{(t-1)^3}\bigl(t^{a_l+\dots+a_{2k-1}}-\\
-t^{a_{k+l}+\dots+a_l}-t^{a_{l+1}+\dots+a_{2k-1}}-t^{a_l+\dots+a_{l+k-1}}+t^{a_{k+l}+\dots+a_{l-1}}+t^{a_1+\dots+a_l}+t^{a_{l+1}+\dots+a_{l+k-1}}-t^{a_1+\dots+a_{l-1}}\bigr).
\end{multline*}

Let us denote $\varphi_j=a_j+\dots+a_{j+k-2}$,
$\psi_j=a_j+\dots+a_{j+k-1}$, $\eta_j=a_1+\dots+a_j$, and
$\kappa_j=a_{j}+\dots+a_{2k-1}$. Then the previous sum is equal to
$$
\frac{1}{(t-1)^3}\left(t^{\eta_l}-t^{\eta_{l-1}}+t^{\kappa_l}-t^{\kappa_{l+1}}+t^{\varphi_{l+k}}-t^{\psi_{l+k}}+t^{\varphi_{l+1}}-t^{\psi_l}\right)
$$

Now let us take the sum over $l=1,\dots,k-1$.

Since $\eta_0=0$, $\kappa_1=n+3$, $\eta_{k-1}=\varphi_1$,
$\kappa_k=\psi_k$, we have:

\begin{prop}
$$
h(P_{a_1,\dots,\,a_{2k-1}})=\frac{1}{(t-1)^3}\left(t^{n+3}-\sum\limits_{i=1}^{2k-1}t^{\psi_i}+\sum\limits_{i=1}^{2k-1}
t^{\varphi_i}-1\right).
$$
\end{prop}

An algebraic meaning of this formula we will see later.

\subsection{Bigraded Betti numbers}
Basic facts and definitions can be found in \cite{BP}.

It is known (see \cite{BP}) that there is an isomorphism of graded
algebras:
$$
\C^*(\mathcal Z_P,\mathbb Q)\cong\Tor_{\mathbb
Z[v_1,\dots,v_m]}(\mathbb Q[P],\mathbb Q)\cong
\C[\Lambda[u_1,\dots,u_m]\otimes\mathbb Q[P],d],
$$
where $du_i=v_i,\, dv_i=0$, $\bideg u_i=(-1,2)$, $\bideg v_i=(0,2)$.

In the last two algebras a full graduation is defined as a sum of
two graduations and the isomorphism between these algebras is an
isomorphism of bigraded algebras.

Thus the cohomology ring of a moment-angle complex has a canonical
bigraded structure arising from an isomorphism with a $\Tor$-algebra
of the Stanley-Reisner ring. This fact gives rise to a series of
combinatorial invariants of a simple polytope
$$
\beta^{-q,\,2p}=\dim \Tor^{-q,\,2p}_{\mathbb
Q[v_1,\dots,v_m]}(\mathbb Q[P],\mathbb Q).
$$

The usual Betti number $\beta^k(\mathcal Z_P)$ is the sum of the
bigraded Betti numbers
$$
\beta^k(\mathcal Z_P)=\sum\limits_{-q+2p}\beta^{-q,\,2p},\;
k=0,1,\dots,m+n.
$$
There is a canonical way to find $\beta^{-q,\,2p}$. Given a graded
$\mathbb Q[v_1,\dots,v_m]$-module $M$ we can build a {\itshape
minimal resolution} of $M$ in the following way. Let us denote
$A=\mathbb Q[v_1,\dots,v_m]$ and let $A^{+}$ be the ideal generated
by all variables.

Let us take the first nonzero graded component of $M$. Consider it
as a linear space and take the basis $m_1,\dots,m_s$. Then let us
generate the submodule $M_1= A<m_1,\dots,m_s>$ and take the minimal
graded component where $M\ne M_1$. Then we find a basis in the
complement to $M_1$ in this component and continue the process. In
the end we obtain a {\itshape minimal basis} or a {\itshape minimal
system of generators} such that the images of it's elements form a
basis of $M\otimes_{A}\mathbb Q$.

Consider a free module $R^0_{\min}$ with generators corresponding to
the minimal basis of $M$. We have a graded epimorphism
$R^0_{\min}\to M$. Then let us take a minimal basis in the kernel of
this map and build a free module $R^{-1}_{\min}$. On the $i$-th step
we take a minimal basis in the kernel of the map
$d:R^{-i+1}_{\min}\to R^{-i+2}_{\min}$ and build $R^{-i}_{\min}$. In
the end we obtain a free resolution of $M$, which is called a
{\itshape minimal resolution}. Since $\mathbb Q$ is a field the
kernel of the map $d:R^{-i}_{\min}\to R^{-i+1}_{\min}$ is a subset
of $A^{+}\cdot R^{-i}_{\min}$, and it's image is a subset of
$A^{+}\cdot R^{-i+1}_{\min}$, $i=1,2,\dots$. So the induced map
$R^{-i}_{\min}\otimes_{A}\mathbb Q\to
R^{-i+1}_{\min}\otimes_A\mathbb Q$ is trivial. Thus we have:
$$
\Tor^{-i}_A(M,\mathbb Q)\cong R^{-i}_{\min}\otimes_{A}\mathbb
Q,\,\mbox{ and } \beta^{-i,\,2j}=\rk R^{-i,\,2j}_{\min}.
$$

Now let us calculate $\beta^{-i,\,2j}$ for simple polytopes with
$n+3$ facets. We know that
$$
\mathbb Q[P_{a_1,\dots,\,a_{2k-1}}]=\mathbb
Q[v_1^1,\dots,v^{a_1}_1,\dots,v_{2k-1}^1,v_{2k-1}^{a_{2k-1}}]/(v_i^1\dots
v_i^{a_i}\dots v_{i+k-2}^1\dots v_{i+k-2}^{a_{i+k-2}})
$$

\begin{itemize}
\item[{\upshape (i)}] According to the previous construction $R^0_{\min}$ is a free
module with a generator $1$ of degree $0$.
\item[{\upshape (ii)}] Generators of $R^{-1}_{\min}$ correspond to
minimal generators of the ideal $(v_i^1\dots v_i^{a_i}\dots
v_{i+k-2}^1\dots v_{i+k-2}^{a_{i+k-2}})$. Thus we have the
generators $X_i$ of degree $2(a_i+\dots+a_{i+k-2})$, which are mapped to \\
$v_i^1\dots v_i^{a_i}\dots v_{i+k-2}^1\dots v_{i+k-2}^{a_{i+k-2}}$,
$i=1,\dots,2k-1$. We see that $2k-1$ is exactly the rank of
$R^{-1}_{\min}$.
\item[{\upshape (iii)}] Now we should find a minimal basis in the
kernel of the map $d : R^{-1}_{\min}\to R^0_{\min}$. Let us denote
$v_i^1\dots v_i^{a_i}$ by $x_i$ and $x_i\dots x_{i+k-2}$ by
$\mathcal V_i$. We claim that that the set
$$
\{x_{i+k-1}X_i-x_iX_{i+1}:\, i=1,\dots,2k-1\}
$$
is a minimal basis. Indeed, it is evident that all this elements lie
in the kernel, are linearly independent and no one of them lies in
the submodule generated by the others.

We need to show that they generate the kernel of $d$. Let
$F=f_1X_1+\dots+f_{2k-1}X_{2k-1}$ and $dF=0$. Then $f_1x_1\dots
x_{k-1}+\dots+f_{2k-1}x_{2k-1}x_1\dots x_{k-2}=0$. All the summands
but $f_ix_i\dots x_{i+k-2}$ are divided either by $x_{i+k-1}$ or by
$x_{i-1}$ so $f_i=p_ix_{i+k-1}+q_ix_{i-1}$. Then
\begin{multline*}
\sum\limits_{i=1}^{2k-1}
f_iX_i=\sum\limits_{i=1}^{2k-1}(p_ix_{i+k-1}+q_ix_{i-1})
X_i=\\
=\sum\limits_{i=1}^{2k-1}p_i(x_{i+k-1}X_i-x_i X_{i+1})+\sum\limits_{i=1}^{2k-1}(p_i+q_{i+1})x_i X_{i+1}=\\
=\sum\limits_{i=1}^{2k-1}p_i(x_{i+k-1} X_i-x_i
X_{i+1})+\sum\limits_{i=1}^{2k-1}f_i'x_iX_{i+1}.
\end{multline*}

The first summand is a combination of generators. Let us consider
the second summand.

$$
\sum\limits_{i=1}^{2k-1}f_i'x_i\mathcal V_{i+1}=0.
$$

All the summands but $f_i'x_i\mathcal V_i$ are divided either by
$x_{i-1}$ or by $x_{i+k}$. So $f_i'=p_i'x_{i+k}+q_i'x_{i-1}$. Then
\begin{multline*}
\sum\limits_{i=1}^{2k-1}f_i'x_iX_{i+1}=\sum\limits_{i=1}^{2k-1}(p_i'x_{i+k}+q_i'x_{i-1})x_iX_{i+1}=\\
=\sum\limits_{i=1}^{2k-1}p_i'x_i(x_{i+k}X_{i+1}-x_{i+1}X_{i+2})+\sum\limits_{i=1}^{2k-1}(p_i'+q_{i+1}')x_ix_{i+1}X_{i+2}.
\end{multline*}

In general case for $s<k-1$ we have:

If $d\left(\sum\limits_{i=1}^{2k-1}\hat f_ix_ix_{i+1}\dots
x_{i+s-1}X_{i+s}\right)=0$, then $\sum\limits_{i=1}^{2k-1}\hat
f_ix_ix_{i+1}\dots x_{i+s-1}\dots x_{i+s+k-2}=0$ and all the
summands but $\hat f_ix_i\dots x_{i+s+k-2}$ are divided either by
$x_{i-1}$ or by $x_{i+s+k-1}$, so $\hat f_i=\hat p_ix_{i+s+k-1}+\hat
q_ix_{i-1}$. Then
\begin{multline*}
\sum\limits_{i=1}^{2k-1}\hat f_ix_i\dots
x_{i+s-1}X_i=\sum\limits_{i=1}^{2k-1}(\hat p_ix_{i+s+k-1}+\hat
q_ix_{i-1})x_i\dots x_{i+s-1}X_{i+s}=\\
=\sum\limits_{i=1}^{2k-1}\hat p_ix_i\dots
x_{i+s-1}(x_{i+s+k-1}X_{i+s}-x_{i+s}X_{i+s+1})+(\hat p_i+\hat
q_{i+1})x_i\dots x_{i+s}X_{i+s+1}
\end{multline*}

In the case $s=k-1$ we obtain
$$
\sum\limits_{i=1}^{2k-1}\hat f_ix_i\dots x_{i+2k-3}=0
$$
Then
$$
\hat f_i=\hat p_i x_{i-1}\quad\mbox{ and
}\quad\sum\limits_{i=1}^{2k-1}\hat p_i=0
$$
So we have:
\begin{multline*}
\sum\limits_{i=1}^{2k-1}\hat f_ix_i\dots x_{i+k-2}
X_{i+k-1}=\sum\limits_{i=1}^{2k-1}\hat p_ix_{i-1}\dots
x_{i+k-2}X_{i+k-1}=\\
=\hat p_1 x_1\dots x_{k-1}(x_{2k-1}X_k-x_kX_{k+1})+(\hat p_1+\hat
p_2)x_2\dots x_k (x_1X_{k+1}-x_{k+1}X_{k+2})+\dots+\\
+(\hat p_1+\dots+\hat p_{2k-2})x_{2k-2}\dots
x_{k-3}(x_{2k-3}X_{k-2}-x_{k-2}X_{k-1}).
\end{multline*}

So our claim is proved.

Thus $R^{-2}_{\min}$ is generated by $Y_i$, $i=1,\dots,2k-1$,
$\bideg Y_i=(-2,2(a_i+\dots+a_{i+k-1}))$,\\
$dY_i=x_{i+k-1}X_i-x_iX_{i+1}$.

\item[{\upshape (iv)}] Let us calculate $R^{-3}_{\min}$. We should find a kernel of
the map $d:R^{-2}_{\min}\to R^{-1}_{\min}$.\\ Let
$d\left(g_1Y_1+\dots+g_{2k-1}Y_{2k-1}\right)=0$. Then
$$
g_1(x_k X_1-x_1X_2)+\dots+g_{2k-1}(x_{k-1}X_{2k-1}-x_{2k-1}X_1)=0.
$$
$$
\begin{matrix}
g_1x_k-g_{2k-1}x_{2k-1}=0;\\
g_2x_{k+1}-g_1x_1=0;\\
\dots\\
g_{2k-1}x_{k-1}-g_{2k-2}x_{2k-2}=0.\\
\end{matrix}
$$

Then we have:
\begin{multline*}
g_1=\frac{x_{k+1}}{x_1}g_2=\frac{x_{k+1}x_{k+2}}{x_1x_2}g_3=\dots=\frac{x_{k+1}x_{k+2}\dots
x_{2k-1}}{x_1x_2\dots x_{k-1}}g_k=\\
=\frac{x_{k+1}\dots x_{2k-1}}{x_2\dots
x_k}g_{k+1}=\frac{x_{k+2}\dots x_{2k-1}}{x_3\dots
x_k}g_{k+2}=\dots=\frac{x_{2k-1}}{x_k}g_{2k-1}
\end{multline*}

So
$$
\begin{matrix}
g_1=ax_{k+1}x_{k+2}\dots x_{2k-1};\\
g_2=ax_{k+2}x_{k+3}\dots x_1;\\
\dots\\
g_{2k-1}=ax_kx_{k+1}\dots x_{2k-2};\\
\end{matrix}
$$
where $a\in\mathbb
Q[v_1^1,\dots,v_1^{a_1},\dots,v_{2k-1}^1,\dots,v_{2k-1}^{a_{2k-1}}]$.

So the kernel of $d$ is generated by $x_{k+1}\dots
x_{2k-1}Y_1+\dots+x_k\dots x_{2k-2}Y_{2k-1}$.

Thus $R^{-3}_{\min}$ has rank $1$ and one generator $Z$ of degree
$2(a_1+\dots+a_{2k-1})=2(n+3)$. This generator is mapped by $d$ to
$x_{k+1}\dots x_{2k-1}Y_1+\dots+x_k\dots x_{2k-2}Y_{2k-1}$.

\item[{\upshape (v)}] The map $d: R^{-3}_{\min}\to R^{-2}_{\min}$ is an injection so
the minimal resolution of $\mathbb Q[P_{a_1,\dots,\,a_{2k-1}}]$ is
constructed.
\end{itemize}

\begin{prop}
For the polytope $P_{a_1,\dots,\,a_{2k-1}}$ we have:
$$
\begin{matrix}
\beta^{0,\,0}=1;\\
\beta^{-1,\,2j}=|\{l: a_l+\dots+a_{l+k-2}=j\}|;\\
\beta^{-2,\,2j}=|\{l: a_l+\dots+a_{l+k-1}=j\}|;\\
\beta^{-3,\,2(n+3)}=1.\\
\end{matrix}
$$
All other bigraded Betti numbers are equal to $0$.
\end{prop}

Now we see that the formula for the $h$-polynomial of
$P_{a_1,\dots,\,a_{2k-1}}$ is exactly the well-known formula (see,
for example, \cite{BP}):
$$
\sum\limits_{p=0}^m\left(\sum\limits_{q=0}^m(-1)^q\beta^{-q,\,2p}\right)t^{2p}=(1-t^2)^{m-n}(h_0+h_1t^2+
\dots+h_nt^{2n})
$$

\subsection{Proof of the main theorem}
Now it is time to prove the main theorem:
\begin{proof}
\begin{enumerate}
\item Without loss of generality we can assume that
$P=\{\ib{x}\in\mathbb R^n:\ib{a}_i\ib{x}_i+b_i\geqslant
0,\,i=1,\dots,m\}$ and $Q=\{\ib{x}\in\mathbb
R^n:\ib{a}_i\ib{x}+b_i\geqslant 0,\,i=1,\dots,m-1\}$. Let $H$ be a
subgroup of $\dim H=s(Q)$ defined by the matrix $M_{m-1\times s(Q)}$
corresponding to $Q$. Then let us take the matrix $\hat
M=\begin{pmatrix} M\\ \mathbb O\end{pmatrix}$ with the line $\mathbb
O=(0,\dots,0)$ corresponding to the new facet. We claim that the
matrix $\hat M$ defines a subgroup that acts freely on $\mathcal
Z_P$.

Consider some vertex $\ib{v}=F_{i_1}\cap\dots\cap F_{i_n}$ of $P$.
If $\ib{v}$ is a vertex of $Q$, then the columns of the matrix
$M\setminus\{\ib{m}_{i_1},\dots,\ib{m}_{i_n}\}$ form a part of a
basis of  $\mathbb Z^{m-1-n}$. Then the columns of the matrix $\hat
M\setminus\{\ib{m}_{i_1},\dots,\ib{m}_{i_n}\}$ form a part of a
basis of $\mathbb Z^{m-n}$.

If $\ib{v}$ lies in $F_m$, then there is an edge of $Q$ that
contains $\ib{v}$. This edge connects $\ib{v}$ with the vertex
$\ib{w}$ which is common for $P$ and $Q$. Let
$\ib{w}=F_{i_0}\cap\dots\cap F_{i_{n-1}}$ and
$\ib{v}=F_{i_1}\cap\dots\cap F_{i_{n-1}}\cap F_m$. Since the columns
of the matrix $M\setminus\{\ib{m}_{i_0},\dots,\ib{m}_{i_{n-1}}\}$
form a part of a basis of $\mathbb Z^{m-1-n}$, the columns of $\hat
M\setminus\{\ib{m}_{i_1},\dots,\ib{m}_{i_{n-1}},\ib{m}_m\}=M\setminus\{\ib{m}_{i_1},\dots,\ib{m}_{i_{n-1}}\}$
form a part of a basis of $\mathbb Z^{m-n}$.

\item It is not difficult to see that any simple $n$-polytope with $n+2$ facets is
projectively equivalent to some product $\Delta^i\times
\Delta^{n-i}$. Indeed, let us take some vertex $\ib{v}$ of $P$ and
consider $n$ facets $F_1,\dots F_n$ intersecting in $v$. Then there
are $n$ edges $e_1,\dots, e_n$ intersecting in $\ib{v}$. Up to a
projective transformation we can assume that the rest two facets
intersect the rays which start from $\ib{v}$ in directions of $e_i$.
Then we can apply another projective transformation to obtain that
each ray intersect one of the facets $F_{n+1}$ and $F_{n+2}$ in the
finite part, and the other -- at infinity. Then there are $i$
``finite'' points in $F_{n+1}$ and $j=n-i$ ``finite'' point in
$F_{n+2}$. Thus after projective transformations
$P=\Delta^i\times\Delta^{n-i}$.

For any polytope we can take the diagonal subgroup defined by the
matrix $\begin{pmatrix}1\\1\\ \vdots\\1\end{pmatrix}$. Thus
$s(\Delta^n)=1$. For the product of simplices
$\Delta^i\times\Delta^{n-i}$ we can take the matrix
$\begin{pmatrix}1&0\\ \vdots\\1&0\\ 0&1\\ \vdots\\0&1\end{pmatrix}$,
which has the form $\begin{pmatrix} M_1 & 0\\0 & M_2\end{pmatrix}$.
Thus $s(\Delta^i\times\Delta^{n-i})=2$.

Since any simple $n$-polytope with with $m\geqslant n+2$ facets
after a projective transformation can be transformed to a polytope
with $n+2$ facets by forgetting the inequalities, we obtain:
$s(P)\geqslant 2$ for $m\geqslant n+2$. So $s(P)=1$ if and only if
$P=\Delta^n$.

\item $s(P)\geqslant 3$ if and only if there exist an $m\times 3$
matrix $M$ consisting of $0$ and $1$ such that for any vertex
$\ib{v}=F_{i_1}\cap\dots\cap F_{i_n}$ the submatrix
$M\setminus\{\ib{m}_{i_1},\dots,\ib{m}_{i_n}\}$ has rank $3$ over
$\mathbb Z_2$.

There are $7$ different nonzero vectors in $\mathbb Z_2^3$.

Let us note that if $a\geqslant7(b-1)+1$ then in any box containing
$a$ balls coloured in $7$ colours there are $b$ balls of the same
colour.

Now let us consider the polytope $P=C^{m-k}(m)^*$ dual to the cyclic
polytope. Let us remind that the facets of $P$ have a canonical
ordering and the subset $\sigma=\{i_1,\dots,i_k\}$ defines the
vertex $\ib{v}=F_{j_1}\cap\dots\cap F_{j_{m-k}}$, where
$\omega=\{j_1,\dots,j_{m-k}\}=[m]\setminus\sigma$, if and only if
there is an even number of point of the set $\omega$ between any two
point of $\sigma$. Thus $\sigma$ should consist of $k$ points such
that any two consequent points from $\sigma$ have different
evenness.

Let $m\geqslant 2(7(7([\frac{k}{2}]-1)+1)+1)-1$. Then there are at least
$7(7([\frac{k}{2}]-1)+1)+1$ odd facets. Let $M$ be a $m\times 3$ matrix
consisting of $0$ and $1$. Without loss of generality we can assume that
there are no zero row vectors. Then there are at least
$7([\frac{k}{2}]-1)+2$ equal row vectors corresponding to odd facets.
There are $7([\frac{k}{2}]-1)+1$ even facets between them. Then there are
at least $[\frac{k}{2}]$ equal row vectors corresponding to even facets.
So we can construct a sequence of $k$ row vectors such that any two
adjacent rows have different evenness and there are at most two different
vectors from $\mathbb Z_2^3$ between them.

Thus if
$m\geqslant2(7(7(\frac{k}{2}-1)+1)+1)-1=7(7(k-2)+2)+2-1=49k-83$,
then $s(C^{m-k}(m))^*)=2$. Since $k=m-n$, we have:

{\itshape If $n+2\leqslant m\leqslant\frac{49}{48}n+\frac{83}{48}$,
then}
$$
s(C^n(m)^*)=2.
$$

\item Let $P$ be an $n_1$-polytope with facets $F_1,\dots,F_{m_1}$
and $Q$ be an $n_2$-polytope with facets $G_1,\dots, G_{m_2}$.

If $M_1$ and $M_2$ are matrices defining subgroups $H_1$ and $H_2$
of dimensions $s(P)$ and $s(Q)$ for $P$ and $Q$ respectively, then
it
is not difficult to see, that the matrix $\begin{pmatrix} M_1& 0\\
0 & M_2\end{pmatrix}$ defines a subgroup for $P\times Q$ that acts
freely, if $H_1$ and $H_2$ act freely. Thus $s(P)+s(Q)\leqslant
s(P\times Q)$.

Now let $H$ be a subgroup that is defined by the matrix $M$ and acts
freely on $\mathcal Z_{P\times Q}=\mathcal Z_P\times\mathcal Z_Q$.
Let us take some vertex $\ib{u}\times\ib{v}$, where without loss of
generality we can assume that  $\ib{u}=F_1\cap\dots\cap F_{n_1}$ is
a vertex of $P$ and $\ib{v}=G_1\cap\dots\cap G_{n_2}$ is a vertex of
$Q$.

Let us take the row vector $\ib{m}_{n_1+1}$ corresponding to
$F_{n_1+1}$. By addition and subtraction of columns we can transform
$\ib{m}_{n_1+1}$ to a vector that has at most one nonzero
coordinate, say the first. Then we can take $\ib{m}_{n_1+2}$
corresponding to $F_{n_1+2}$ and transform it to a vector that has
at most one nonzero coordinate among the coordinates $2,\dots,s$,
say the second. Iterating this process in the end we obtain the
matrix
$\hat M=\begin{pmatrix}M_{11}& M_{12}\\
M_{13}& 0\\ A& M_2\end{pmatrix}$ with at most $m_1-n_1$ nonzero
columns in the part $M_{13}$ corresponding to the facets
$F_{n_1+1},\dots,F_{m_1}$ of $P$.  Since for any vertex $\ib{w}$ of
$Q$ the product $\ib{u}\times\ib{w}$ is a vertex of $P\times Q$ we
see that $M_2$ defines a subgroup that act freely on $\mathcal
Z_Q$.\\ So $s(Q)\geqslant s(P\times Q)-(m_1-n_1)$. Similarly
$s(P)+(m_2-n_2)\geqslant s(P\times Q)$.

Thus $s(P\times Q)\leqslant
s(P)+s(Q)+\min\{m_1-n_1-s(P),m_2-s_2-s(Q)\}$.

\item Let $M_1$ be a matrix defining a subgroup $H_1$ of $\dim
H_1= s(P)$ that acts freely on $\mathcal Z_P$ and $M_2$ -- $H_2$ of
$\dim H_2=s(Q)$ that acts freely on $\mathcal Z_Q$. Let $P\sharp Q$
be a connected sum of $P$ and $Q$ along the vertices
$\ib{u}=F_{m_1-n+1}\cap\dots\cap F_{m_1}$ and
$\ib{v}=G_1\cap\dots\cap G_n$. Let $M_1=\begin{pmatrix}M_{11}\\
M_{12}\end{pmatrix}$ and
$M_2=\begin{pmatrix}M_{21}\\M_{22}\end{pmatrix}$, where $M_{12}$
corresponds to $\ib{u}$ and $M_{21}$ -- to $\ib{v}$. Then we claim
that the matrix $M=\begin{pmatrix}M_{11}& 0\\ M_{12}& M_{21}\\0&
M_{22}\end{pmatrix}$ defines a subgroup that acts freely on
$\mathcal Z_{P\sharp Q}$. Indeed, any vertex of $P\sharp Q$ is
either a vertex of $P$ or a vertex of $Q$. Let $\ib{w}$ be a vertex
of $P\sharp Q$ and $P$. Then $\ib{w}=F_{i_1}\cap\dots\cap F_{i_n}$.
Then the columns of the matrix
$M_1\setminus\{\ib{m}_{1,i_1},\dots,\ib{m}_{1,i_n}\}$ form a part of
a basis of $\mathbb Z^{m_1-n}$. But the columns of the matrix
$M_{22}$ form a part of a basis of $\mathbb Z^{m_2-n}$, since
$\ib{v}$ is a vertex of $Q$. Then the columns of the matrix
$M\setminus\{\ib{m}_{i_1},\dots,\ib{m}_{i_n}\}$ form a part of a
basis of $\mathbb Z^{(m_1+m_2-n)-n}=\mathbb Z^{m_1-n}\oplus\mathbb
Z^{m_2-n}$. Thus $s(P)+s(Q)\leqslant s(P\sharp Q)$.

\item This property is evident if we use the second combinatorial
description, since any map $\varphi$ from the set of vertices of
$K_2$ to $m_2-s(K_2)$ such that for any $(n-1)$-simplex of $K_2$ the
corresponding vectors form a part of a basis of $\mathbb
Z^{m_2-s(K_2)}$ induces the map $\varphi f$ with the same property.

\item An inclusion of the face corresponding to $F$ under the polar
transformation is a non-degenerate map.

\item It follows from the fact that a right coloring of the vertices
of the $n$-dimensional simplicial polytope $P^*$ is exactly a
non-degenerate map $P^*\to \Delta^{\gamma-1}_{n-1}$. Now let us
prove that for any $(n-1)$-dimensional simplicial complex $K$ with
$m$ facets $s(K)\geqslant\left[\frac{m}{n+1}\right]$. Indeed, let us
construct the matrix
$$
C=\begin{pmatrix}
M\quad\mathbb O\\
I_{na+t}
\end{pmatrix}
$$
where $M=\begin{pmatrix}I_a&I_a&\ldots&I_a\end{pmatrix}$ consists of
$n$ blocks $I_a$ and $I_s$ is an identity matrix of sizes $s\times
s$.

If $(n+1)a+t=m$, then each $n$ row vectors of the matrix $C$ form a part
of a basis of $\mathbb Z^{an+t}$, since if they lie in the
$I_{na+t}$-part it is evident, and if $p>0$ vectors
$\ib{a}_{i_1},\dots,\ib{a}_{i_p}$ lie in the $M$-part and $q$ row vectors
$\ib{a}_{j_1},\dots,\ib{a}_{j_q}$ -- in the $I_{na+t}$-part, $p+q=n$,
then $q<n$ and we can find the block $I_a$ of the matrix $M$ such that
the corresponding columns of the matrix $C$ have zero components in
$\ib{a}_{j_l}$ for all $l=1,\dots,q$. Then these $n$ vectors form a part
of a basis of $\mathbb Z^{na+t}$. So $m-s(K)\leqslant an+t$ and
$s(K)\geqslant a$.

At last, let us take $a=\left[\frac{m}{n+1}\right]$. So we have this
estimate for $s(\Delta^{\gamma-1}_{n-1})$. The second inequality
follows from the fact that the function
$m-\gamma+\left[\frac{\gamma}{n+1}\right]$ decreases in the variable
$\gamma$.

\item As it was mentioned above in the cases of matrices $H$ of
sizes $m\times 2$ and $m\times 3$ it is sufficient to work over the
field $\mathbb Z_2$. Since substitution of any row vector for a zero
row vector doesn't decrease rank of a matrix, we can assume that all
row vectors are nonzero. Let $m-n=k$.
\begin{itemize}
\item[a)] $s(\Delta^{m-1}_{n-1})\geqslant 2$ if and only if there
exists an $m\times 2$ matrix $H$ consisting of $0$ and $1$ such that
any collection of $k$ vectors has rank $2$, that is consists of more
than one different vectors. Let $s_1,\,s_1,\,s_3$ be the number of
the row vectors $(1,0),\,(0,1)$, and $(1,1)$ respectively. Then
$s_i\leqslant k-1$, so $m=s_1+s_2+s_3\leqslant 3(k-1)=3(m-n-1)$. So
$m\geqslant\frac{3}{2}(n+1)$. On the other hand, if
$m<\frac{3}{2}(n+1)$, then $m>3(k-1)$. So one of the numbers $s_i$
is greater than $k-1$ and there are $k$ equal row vectors in the
matrix $H$.

\item[b)] $s(\Delta^{m-1}_{n-1})\geqslant 3$ if and only if there
exists an $m\times 3$ matrix $H$ consisting of $0$ and $1$ such that
the rank of any collection of $k$ vectors is equal to $3$. This
condition is equivalent to the existence of a $3\times 3$-submatrix
with the determinant equal to $1$, that is, the corresponding row
vectors should be pairwise different and their sum should not be
equal to zero. For any two different nonzero vectors in $\mathbb
Z_2^3$ there exists exactly one third nonzero vector such that they
are linearly dependent -- it is their sum. Then there are
$\frac{{7\choose 2}}{3}=7$ different linearly dependent triples of
different nonzero vectors.
$$
\begin{matrix}
&(1,0,0)&\times&\times&\times&&&&\\
\bullet&(0,1,0)&\times&&&\times&\times&&\\
\bullet&(0,0,1)&&\times&&\times&&\times&\\
\bullet&(1,1,0)&\times&&&&&\times&\times\\
\bullet&(1,0,1)&&\times&&&\times&&\times\\
&(0,1,1)&&&\times&\times&&&\times\\
&(1,1,1)&&&\times&&\times&\times&\\
\end{matrix}
$$
Let $s_i$ be the number of the $i$-th row vector in the matrix $H$.
For any triple of the linearly dependent vectors $(i_1,i_2,i_3)$ we
should have: $s_{i_1}+s_{i_2}+s_{i_3}\leqslant k-1$. Then
$3m\leqslant 7(k-1)$. It is a necessary condition.
\begin{itemize}

\item If $m=7l$, we can take $s_i=l$, so if the previous estimate is
valid, then
$$
s_{i_1}+s_{i_2}+s_{i_3}=3l\leqslant k-1
$$
and thus the estimate is also sufficient.

\item If $m=7l+1$, then we can take $s_1=l+1$, and $s_i=l,\, i>1$.
For any triple of linearly dependent vectors we have
$s_{i_1}+s_{i_2}+s_{i_3}\leqslant 3l+1$. If $3l+1\leqslant k-1$,
then we can build a matrix $H$. On the other hand, if $3m\leqslant
7(k-1)$, then $3l+\frac{3}{7}\leqslant k-1$, so $3l+1\leqslant k-1$,
since all the numbers are integer. Thus the condition $3l+1\leqslant
k-1$ is necessary and sufficient.

\item If $m=7l+a,\,a=2,3$, or $4$, then we can take some of the
vectors marked by the black points $l+1$ times and all the others
$l$ times. If $3l+2\leqslant k-1$, then we can build $H$. On the
other hand, let $3l+2\geqslant k$. Let $s_i=\max\{s_1,\dots,s_7\}$.
Then $s_i\geqslant l+1$. Without loss of generality we can take
$i=1$. One of the sums $s_2+s_4,\,s_3+s_5,\,s_6+s_7$ is greater than
or equal to $\frac{7l+a-s_1}{3}$, therefore if we add $s_1$, then we
obtain
$$
s_1+s_{i_2}+s_{i_3}\geqslant\frac{7l+a+2s_1}{3}\geqslant\frac{7l+2+2(l+1)}{3}=\frac{9l+4}{3}=3l+\frac{4}{3}.
$$
Since the number on the left side is integer, we have
$s_1+s_{i_2}+s_{i_3}\geqslant 3l+2\geqslant k$. This is an obstacle
to the existence of $H$.

Hence for $m=7l+2,\,7l+3,\,7l+4$, $s(P)\geqslant 3$ if and only if
$3l+2\leqslant k-1$.

\item If $m=7l+b,\, b=5,6$, then we can take $b$
vectors $l+1$ times and all the other vectors $l$ times. If
$3l+3\leqslant k-1$, then we can build $H$. Otherwise one of the
seven sums is greater than or equal to
$\frac{3m}{7}=3l+\frac{3b}{7}\geqslant 3l+\frac{15}{7}>3l+2\geqslant
k-1$. Then it is greater than or equal to $k$, and we obtain an
obstacle to the existence of $H$.
\end{itemize}
If we substitute $m-n$ for $k$ then we obtain the result of the
statement.

\end{itemize}

\item Let $\omega_1,\dots,\omega_l$ be minimal non-simplices such
that $\omega_1\cup\dots\cup\omega_l=[m]$. Then for each
$\omega_i=(j_{i,1},\dots,j_{i,|\omega_i|})$ let us build the
$m\times\dim\omega_i$-matrix $M_i$ with row vectors:
$$
\ib{m}_{i,j}=\begin{cases} (0,\dots,0),&j\notin\omega;\\
(\underbrace{0,\dots,0}_{t-1},1,0,\dots,0),&j=j_{i,t}\in\{j_{i,1},\dots,j_{i,|\omega_i|-1}\};\\
(1,\dots,1),&j=j_{i,\,|\omega_i|}.
\end{cases}
$$

Then we claim that the matrix $C=(M_1,\dots,M_l)$ has the properties
we need. In fact, the set $(i_1,\dots,i_n)$ that defines a maximal
simplex $\sigma$ (or a vertex of polytope $P$ in the case of
$K=\partial P^*$) does't contain any $\omega_i$. So let us consider
all the vertices of $\sigma$ in $\omega_1$. The corresponding row
vectors form a part of a basis of $\mathbb Z^{\dim\omega_1}$ and, in
fact, by an addition and a subtraction of the columns in $M_1$ we
can make this vectors a part of the standard basis. Then by an
addition and a subtraction of the columns in $C$ we can make the
rest coordinates corresponding to $M_2,\dots,M_l$ equal to $0$.

Then let us take $\omega_2\setminus\omega_1$ and do the same
operation. Then we take $\omega_3\setminus(\omega_1\cup\omega_2)$,
and so on. In the end we obtain that the vectors
$\ib{c}_{i_1},\dots.\ib{c}_{i_n}$ form a part of the standard basis
of $\mathbb Z^{\dim\omega_1+\dots+\dim\omega_l}$. Therefore
$m-s(K)\leqslant\dim\omega_1+\dots+\dim\omega_l$.

\item It is easy to see that a polytope is $k$-flag if and only if
any minimal non-simplex has dimension $\leqslant k-1$. Then a flag
polytope has all minimal simplices -- edges. In this case, since any
set of more than $n$ facets has the empty intersection, it should
contain some pair of facets that don't intersect.

Therefore let us start with $S_1=[m]$ and take two facets that don't
intersect, say $F_{i_1}\cap F_{i_2}=\varnothing$. Then let us colour
them in the colour $1$ and take $S_2=S\setminus\{i_1,i_2\}$. We can
take this step until $|S_i|>n$. So if $m-n=2k+1$, then we can take
$k+1$ steps and colour the rest $n-1$ facets in $n-1$ additional
colour. If $m-n=2k$, then we stop after $k$ steps and colour the
rest facets in additional $n$ colours. In both cases we have
$\gamma(P)\leqslant\left[\frac{m-n}{2}\right]+n$. Thus
$s(P)\geqslant
\left\lceil\frac{m-n}{2}\right\rceil+s(\Delta^{\gamma-1}_{n-1})$.

\item In the case of $k$-flag polytopes we can not colour in one colour facets
corresponding to a minimal non-simplex of dimension greater than
$1$. But again start with $S_1=[m]$ and choose some minimal
non-simplex $\omega_1\subset S_1$. Let $S_2=S_1\setminus \omega_1$.
In the end we take $r$ steps and come to the situation, when
$|S_{r+1}|\leqslant n$. Let
$\Sigma=|\omega_1\sqcup\dots\sqcup\omega_r|$. Then we can take
additional $m-\Sigma$ minimal non-simplices of dimensions at most
$k-1$. Thus we have: $m-s(P)\leqslant (\Sigma-r)+(m-\Sigma)(k-1)$.
Therefore $s(P)\geqslant r-(m-\Sigma)(k-2)$. But
$r\geqslant\left\lceil\frac{m-n}{k}\right\rceil$, and
$m-\Sigma\leqslant n$. So
$$
s(P)\geqslant\left\lceil\frac{m-n}{k}\right\rceil-n(k-2).
$$
In particular, since any polytope $P$ except for $\Delta^n$ is
$n$-flag, we have
$$
s(P)\geqslant\left\lceil\frac{m-n}{n}\right\rceil-n(n-2).
$$

\item If $M=\{\ib{m}_1,\dots,\ib{m}_m\}$ is an $m\times s$ matrix defining a
subgroup $H$ of dimension $s(P)$ interpreted as a set of
row-vectors, then it is easy to see, that the matrix $\hat
M=\{\underbrace{\ib{m}_1,\dots,\ib{m}_1}_{k_1},\dots,\underbrace{\ib{m}_m,\dots,\ib{m}_m}_{k_m}\}$
defines the subgroup $\hat H$, which acts freely if $H$ acts freely.

On the other hand, let $\hat H$ be a subgroup of $\dim
H=s(P_{k_1,\dots,\,k_m})$ defined by the matrix
$$
\hat
M=\{\ib{m}^1_1,\dots,\ib{m}^{k_1}_1,\dots,\ib{m}^1_m,\dots,\ib{m}^{k_m}_m\}.
$$
Then the matrix $M=\{\ib{m}^1_1,\dots,\ib{m}^1_m\}$ gives a subgroup
for $P$.

\item $P_{a_1,a_2,\,a_3}=\Delta^{a_1-1}\times\Delta^{a_2-1}\times\Delta^{a_3-1}$.
For $k\geqslant 3$ it is enough to consider the cyclic polytopes
$P_k=C^{2k-4}(2k-1)^*$, according to the Corollary 6. For $k=3$ the
polytope $P_3$ is a usual $5$-gon. For $k=4$ let us find all
possible matrices $M=\{\ib{m}_1,\dots,\ib{m}_7\}$ consisting of $0$
and $1$ that define subgroups $H$ of $\dim H=3$ acting freely on
$\mathcal Z_{P_4}$. Without loss of generality we can assume that
$$
\ib{m}_1=\ib{e}_1=\begin{pmatrix}1\\0\\0\end{pmatrix},\,
\ib{m}_4=\ib{e}_2=\begin{pmatrix}0\\1\\0\end{pmatrix},\,
\ib{m}_5=\ib{e}_3=\begin{pmatrix}0\\0\\1\end{pmatrix}.
$$

The rest four vectors should be
$$
\begin{pmatrix}1\\1\\0\end{pmatrix},\,
\begin{pmatrix}1\\0\\1\end{pmatrix},\,\begin{pmatrix}0\\1\\1\end{pmatrix},\mbox{
and }\begin{pmatrix}1\\1\\1\end{pmatrix}.
$$

Since the triples $(\ib{m}_1,\ib{m}_2,\ib{m}_5)$,
$(\ib{m}_1,\ib{m}_3,\ib{m}_5)$, $(\ib{m}_1,\ib{m}_4,\ib{m}_6)$, and
$(\ib{m}_1,\ib{m}_4,\ib{m}_7)$ define the vertices, $\ib{m}_2$ and
$\ib{m}_3$ should have the form $\begin{pmatrix}*
\\ 1\\ *\end{pmatrix}$, and $\ib{m}_6,\ib{m}_7$ --
$\begin{pmatrix}* \\ *\\ 1\end{pmatrix}$.

Let us remind that for the $3\times3$-matrix consisting of $0$ and
$1$ it's determinant is equal to $\pm 1$ if and only of it's row
vectors are linearly independent over $\mathbb Z_2$.

Then since $(\ib{m}_{2},\ib{m}_3,\ib{m}_6)$ and
$(\ib{m}_3,\ib{m}_6,\ib{m}_7)$ should be linearly independent, none
of this triple is equal to $\left(
\begin{pmatrix}1\\1\\0\end{pmatrix},
\begin{pmatrix}1\\0\\1\end{pmatrix},
\begin{pmatrix}0\\1\\1\end{pmatrix}\right)$. Thus either $\ib{m}_3=\begin{pmatrix}1\\1\\1\end{pmatrix}$ or
$\ib{m}_6=\begin{pmatrix}1\\1\\1\end{pmatrix}$.

In the first case $\ib{m}_2=\begin{pmatrix}1\\1\\0\end{pmatrix}$,
and since $(\ib{m}_3,\ib{m}_4,\ib{m}_7)$ defines the vertex,
$\ib{m}_7=\begin{pmatrix}0\\1\\1\end{pmatrix}$,
$\ib{m}_6=\begin{pmatrix}1\\0\\1\end{pmatrix}$. It is not difficult
to check that the triples $(\ib{m}_2,\ib{m}_5,\ib{m}_6)$,
$(\ib{m}_2,\ib{m}_5,\ib{m}_7)$, $(\ib{m}_2,\ib{m}_4,\ib{m}_7)$,
$(\ib{m}_1,\ib{m}_3,\ib{m}_6)$, $(\ib{m}_2,\ib{m}_4,\ib{m}_6)$, and
$(\ib{m}_3,\ib{m}_5,\ib{m}_7)$ are linearly independent over
$\mathbb Z_2$. Thus all $14$ conditions corresponding to the
vertices hold.

In the second case $\ib{m}_2=\begin{pmatrix}0\\1\\1\end{pmatrix}$,
$\ib{m}_3=\begin{pmatrix}1\\1\\0\end{pmatrix}$,
$\ib{m}_7=\begin{pmatrix}1\\0\\1\end{pmatrix}$.

Since for any two facets of $P_k$ there exists a vertex that doesn't
belong to them, and there are $7$ different nonzero vectors in
$\mathbb Z_2^3$, for $2k-1\geqslant 9$ there are no subgroups of
dimension $3$ acting freely on $\mathcal Z_{P_k}$.

So we proved that $s(P_{a_1,\dots,\,a_{2k-1}})=3$ if and only if
$k\leqslant 4$.

From Proposition 10 we obtain that
$$
2k-1=\sum\limits_{j}\beta^{-1,\,2j}(\mathcal
Z_{P_{a_1,\dots,\,a_{2k-1}}})=\sum\limits_{j}\beta^{-2,\,2j}(\mathcal
Z_{P_{a_1,\dots,\,a_{2k-1}}}).
$$

\item Let us consider two polytopes $7$-dimensional neighbourly polytopes with $10$ facets
$$
P=(2,1,1,1,1,1,1,1,1)\mbox{ and } Q=(2,1,2,1,1,2,1).
$$
$Q$ is obtained from $P$ by $4=\frac{7+1}{2}$ -flip. Then
$f(P)=f(Q),\gamma(P)=\gamma(Q)$, but $s(P)=2$ and $s(Q)=3$.
Nevertheless $\sum\limits_i\beta^{-1,\,2i}(P)=9$, but
$\sum\limits_j\beta^{-1,\,2j}(Q)=7$.

In fact, $P=C^7(10)^*$ and $Q$ is also a polytope dual to a
neighborly polytope.

\item Consider an $i$-flip $2\leqslant i\leqslant n-1$, which transforms the polytope $P^n$ to $Q^n$.
Then there are $n+1$ facets $F_{j_1},\dots,F_{j_{n+1}}$ of $P$ such
that for $t\in I=\{1,\dots i\}$ the facets
$F_{j_1},\dots,\widehat{F_{j_t}}\dots F_{j_{n+1}}$ intersect in a
vertex and for $t\in J=\{i+1,\dots, n+1\}$ this is false.

The flip exchanges the sets $I$ and $J$, so for the polytope $Q$ the set
$J$ plays the role of $I$. All other vertices are the same in $P$ and
$Q$. Let us use the second combinatorial description of the $s$-number.
Then there is the map $\mathfrak F(P)=\{F_1,\dots,F_m\}\to\mathbb
Z^{m-s(P)}: F_j\to\ib{m}_j$. We can build the map $\mathfrak
F(Q)\to\mathbb Z^{m-s(P)}\oplus\mathbb Z$:
$$
F_j\to(\ib{m}_j,0),\,j\ne j_1,\quad F_{j_1}\to(\ib{m}_{j_1},1)
$$

It is easy to see that if the map for $P$ satisfies the condition
that in every vertex $F_{i_1}\cap\dots\cap F_{i_n}$ the vectors
$\ib{m}_{i_1},\dots,\ib{m}_{i_n}$ form a part of a basis of $\mathbb
Z^{m-s(P)}$, then so does the map for $Q$ with respect to $\mathbb
Z^{m-s(P)+1}$. Since $P$ and $Q$ have the same number of facets, we
have:
$$
m-s(Q)\leqslant m-s(P)+1. $$ But the inverse transformation is an
$(n+1-i)$-flip, $2\leqslant n+1-i\leqslant n-1$. So
$|s(P)-s(Q)|\leqslant 1$.

A $1$-flip is just a cutting off the vertex, that is a connected sum
$P\sharp\Delta^n$ with $\Delta^n$ along the vertices. An $n$-flip is an
inverse operation. We know that $s(P)+1=s(P)+s(\Delta^n)\leqslant
s(P\sharp \Delta^n)$. When we make an $n$-flip, we substitute a vertex
$\ib{v}=F_{j_1}\cap\dots F_{j_n}$ for the facet $F_{j_{n+1}}$ -- an
$(n-1)$-dimensional simplex. Then the previous construction gives us the
bound $m-s(P)\leqslant m+2-s(P\sharp\Delta^n)$, where $m$ is the number
of facets of $P$ and $m+1$ -- of $P\sharp\Delta^n$. Thus we have:
$$
s(P)+1\leqslant s(P\sharp\Delta^n)\leqslant s(P)+2
$$
\end{enumerate}
\end{proof}

\begin{remark}
We obtained two lower bounds for $P$:
$s(P)\geqslant\left[\frac{m}{n+1}\right]$ and
$s(P)\geqslant\left\lceil\frac{m-n}{n}\right\rceil-n(n-2)$. In fact,
we can obtain a stronger estimate using the property $9b)$. Is says
that any $\frac{7}{4}(n+1)+2$ facets add to $s(P)$ at least $3$
(since we can build the matrix $C$ as a block matrix). Thus
$$
s(P)\geqslant
3\left[\frac{m}{\lceil\frac{7}{4}(n+1)+2\rceil}\right]>3\left[\frac{4m}{7n+19}\right].
$$
\end{remark}
\subsection{Cohomology ring of $P_{a_1,\dots,\,a_{2k-1}}$}
There is another isomorphism that we need for our computations (see
\cite{BP}):
$$
\C^{*,\,*}(\mathcal Z_P,\mathbb Z)\cong \C[R^*(P)],
$$
where
$$
R^*(P)=\Lambda[u_1,\dots,u_m]\otimes\mathbb
Z[P]/(v_i^2=u_iv_i=0,\,i=1,\dots,m),
$$
$$
\bideg u_i=(-1,2),\,\bideg v_i=(0,2),\quad du_i=v_i,\,dv_i=0.
$$

Let us apply this theorem for the case of
$P=P_{a_1,\dots,\,a_{2k-1}}$.

We will use the following observations:

\textbf{I}. Since $du_i=v_i$ and $dv_i=0$ the rings $R^*(P)$ and
$H[R^*(P)]$ have a multigraded structure (see \cite{BP}): any
monomial
$$
u_{\omega}v_{\sigma}=u_{i_1}\wedge\dots\wedge u_{i_s}v_{j_1}\dots
v_{j_t},\,\omega=\{i_1,\dots,i_s\},\,\sigma=\{j_1,\dots,j_t\},\,\omega\cap\sigma=\varnothing
$$
has graduation $2(\alpha_1,\dots,\alpha_m)$, where $\alpha_i=1$, if
$i\in\tau=\omega\cup\sigma$ and $0$ in the other case. Sometimes we
will denote this graduation by $\tau$. If the element $x$ has
graduation $2(\alpha_1,\dots,\alpha_m)$ and $\bideg x=(-q,2p)$, then
$p=\alpha_1+\dots+\alpha_m=|\tau|$.

\textbf{II}. Let $i\in\tau$. Then for each monomial
$u_{\omega}v_{\sigma}: i\in\sigma$ we have:
$$
du_{\omega\cup\{i\}}v_{\sigma\setminus\{i\}}=\pm
u_{\omega}v_{\sigma}+\sum\limits_{j\in\omega}\pm
u_{\omega\cup\{i\}\setminus\{j\}}v_{\sigma\cup\{j\}\setminus\{i\}}.
$$
Thus up to a coboundary any element
$x\in\Lambda[u_1,\dots,u_m]\otimes\mathbb Z[P]/(v_i^2=u_iv_i=0)$ of
graduation $\tau$ and degree $\bideg x=(-q,2|\tau|)$ has the form
$y\wedge u_i$.

Then $dx=dy\wedge u_i\pm y\cdot v_i$. Thus $dx=0$ if and only if
$dy=0$ and $y\cdot v_i=0$. From this fact it follows that the
cohomology ring in the case of $\Lambda[u_1,\dots,u_m]\otimes\mathbb
Z[v_1,\dots,v_m]/(v_i^2=u_iv_i=0)$ is trivial, since in this case
$yv_i=0$ if and only if $y=0$.

\textbf{III}. If $x\cdot v_i=0$, and $y\cdot v_i=0$, then
$(x+dy)\cdot v_i=0$ and $(x+dy)\wedge u_i=x\wedge u_i+d(y\wedge
u_i)$.

Now let us calculate $\C^{*,\,*}$.
\begin{itemize}
\item[{\upshape (i)}] $\C^{0,\, 0}=\mathbb Z$ with a generator $1$, $\C^{0,\,2i}=0,
i>0$ since $v_{i_1}\dots v_{i_t}=du_{i_1}v_{i_2}\dots v_{i_t}$.

\item[{\upshape (ii)}] According to the second observation any
element of outer degree $-1$ and graduation $\tau$ is equivalent to
the monomial $\lambda v_{i_1}\dots v_{i_{t-1}}u_{i_t}$. Then
$dv_{i_1}\dots v_{i_{t-1}}u_{i_t}=v_{i_1}\dots v_{i_{t-1}}v_{i_t}$.
If it is equal to $0$, then $\tau=\{i_1,\dots,i_t\}$ should contain
the segment $J_i$ corresponding to $(a_i,\dots,a_{i+k-2})$ for some
$i$ and $i_t\in J_i$. If $\tau$ does not coincide with this segment,
then $i_1\in\tau\setminus J_i$. Then $du_{i_1}v_{i_2}\dots
v_{i_{t-1}}u_{i_t}=v_{i_1}\dots v_{i_{t-1}}u_{i_t}$.

So any nontrivial cocycle has graduation $\tau=J_i$ for some $i$ and
is equivalent to the monomial
$\lambda_iv_{\eta_{i-1}+1}v_{\eta_{i-1}+2}\dots
v_{\eta_{i-1}+\varphi_i-1}u_{\eta_{i-1}+\varphi_i}$, where
$[\eta_{i-1}+1,\dots,\eta_{i-1}+\varphi_i]=J_i$ (Let us remind the
notations $\eta_i=a_1+\dots+a_i,\,
\eta_0=0,\,\varphi_i=a_i+\dots+a_{i+k-2},\,
\psi_i=a_i+\dots+a_{i+k-1}$). It is easy to see that such a monomial
is not a coboundary. Indeed, let $\lambda v_{\eta_{i-1}+1}\dots
v_{\eta_{i-1}+\varphi_i-1}u_{\eta_{i-1}+\varphi_i}=dx$. $x$ has the
same graduation, so according to \textbf{II}
$x=\sum\limits_{s=1}^{\varphi_i-1}\lambda_sv_{\eta_{i-1}+1}\dots
u_{\eta_{i-1}+s}\dots
v_{\eta_{i-1}+\varphi_i-1}u_{n_{i-1}+\varphi_i}$. But
\begin{multline*}
dv_{\eta_{i-1}+1}\dots u_{\eta_{i-1}+s}\dots
v_{\eta_{i-1}+\varphi_i-1}u_{\eta_{i-1}+\varphi_i}=v_{\eta_{i-1}+1}\dots
v_{\eta_{i-1}+s}\dots
v_{\eta_{i-1}+\varphi_i-1}u_{\eta_{i-1}+\varphi_i}-\\
-v_{\eta_{i-1}+1}\dots u_{\eta_{i-1}+s}\dots
v_{\eta_{i-1}+\varphi_i-1}v_{\eta_{i-1}+\varphi_i},
\end{multline*}
where both summands are nonzero. So $d\lambda v_{\eta_{i-1}+1}\dots
v_{\eta_{i-1}+\varphi_i-1}u_{\eta_{i-1}+\varphi_i}=ddx=0$ in the
ring $\Lambda[u_1,\dots,u_m]\otimes\mathbb
Z[v_1,\dots,v_m]/(v_i^2=u_iv_i=0)$. But this is false.

Thus we obtain:

{\itshape $\C^{-1,\,*}\cong\mathbb Z^{2k-1}$ with generators
$X_i=v_{\eta_{i-1}+1}\dots
v_{\eta_{i-1}+\varphi_i-1}u_{\eta_{i-1}+\varphi_i}$ of graduation
$J_i$ corresponding to the segments $(a_i,\dots,a_{i+k-2})$ of the
polygon. $\bideg X_i=(-1,2(a_i+\dots+a_{i+k-2}))$.}

\item[{\upshape (iii)}] Any cocycle of outer degree $-2$
and graduation $\tau=(i_1,\dots,i_t)$ up to a coboundary is equal to
$x\wedge u_{i_t}$ where $dx=0$ and $x\cdot v_{i_t}=0$. $\tau$ is not
the full circle, since any $n+1$ facets of our polytope have the
empty intersection. So we can assume that $i_k$ follow each other
one the circle and $i_t+1\notin \tau$. Then there is some $i$ such
that $J_i=[\eta_{i-1}+1,\dots,\eta_{i-1}+\varphi_i]\subseteq\tau$
and $i_t=\eta_{i-1}+\varphi_i$. When we apply the second observation
to $x$ with $i=i_{t-\varphi_i}$, we don't change the $J_i$-part of
$x$. Thus $x=\lambda v_{i_1}\dots
u_{i_{t-\varphi_i}}v_{i_{t-\varphi_i+1}}\dots
v_{i_{t-1}}u_{i_t}+dy$, where $y\cdot v_{i_t}=0$. Then
$$
x\wedge u_{i_t}=\lambda v_{i_1}\dots
u_{i_{t-\varphi_i}}v_{i_{t-\varphi_i+1}}\dots
v_{i_{t-1}}u_{i_t}+(dy)\wedge u_{i_t}=\lambda v_{i_1}\dots
u_{i_{t-\varphi_i}}v_{i_{t-\varphi_i+1}}\dots
v_{i_{t-1}}u_{i_t}+d(y\wedge u_{i_t}).
$$

Since $dx=0$, we have: $v_{i_1}\dots
v_{i_{t-\varphi_i}}v_{i_{t-\varphi_i+1}}\dots v_{i_{t-1}}=0$, so
$J_{i-1}\subset\tau$ and $i_{t-\varphi_i}\in J_{i-1}$. If $\tau\ne
J_{i-1}\cup J_i$, then $i_1\in\tau\setminus (J_{i-1}\cup J_i)$. Then
$$
du_{i_1}v_{i_2}\dots u_{i_{t-\varphi_i}}v_{i_{t-\varphi_i+1}}\dots
v_{i_{t-1}} u_{i_t}=v_{i_1}v_{i_2}\dots
u_{i_{t-\varphi_i}}v_{i_{t-\varphi_i+1}}\dots v_{i_{t-1}}u_{i_t}.
$$
Let us mention that in the case of $\tau=J_{i-1}\cup J_i$ the
position of $u$ in the vertex $i-1$ of the polygon can be chosen
arbitrarily, since if $u_{s_1}$ and $u_{s_2}$ are two positions (for
example, $u_{s_1}=u_{i_{t-\varphi_i}}$), then
\begin{multline*}
dv_{i_1}\dots u_{s_1}\dots u_{s_2}\dots v_{i_{t-\varphi_i}}\dots
v_{i_{t-1}}u_{i_t}=v_{i_1}\dots v_{s_1}\dots u_{s_2}\dots
v_{i_{t-\varphi_i}}\dots v_{i_{t-1}}u_{i_t}-\\
-v_{i_1}\dots u_{s_1}\dots v_{s_2}\dots v_{i_{t-\varphi_i}}\dots
v_{i_{t-1}}u_{i_t}.
\end{multline*}
The same is true for the vertex $i+k-2$.

Thus $\C^{-2,\,*}$ is generated by the monomials
$Y_i=u_{\eta_{i-1}+1}v_{\eta_{i-1}+2}\dots
v_{\eta_{i-1}+\psi_i-1}u_{\eta_{i-1}+\psi_i}$ where
$L_i=[\eta_{i-1}+1,\dots, \eta_{i-1}+\psi_i]$ is the segment
corresponding to the segment $(a_i,\dots,a_{i+k-1})$ on the circle.
All this monomials have different graduation, so let us consider one
of them.

Let $\lambda Y_i=dy,\,\lambda\in\mathbb Z$. Without loss of
generality $y=x\wedge u_{\eta_{i-1}+\psi_i}$.\\ Then $dx=\lambda
u_{\eta_{i-1}+1}v_{\eta_{i-1}+2}\dots v_{\eta_{i-1}+\psi_i-1}$ and
$x\cdot v_{\eta_{i-1}+\psi_i}=0$.

Then each monomial of $x$ contains the segment $v_{\eta_i+1}\dots
v_{\eta_i+\varphi_{i+1}-1}$. When we apply \textbf{II} for
$i=\eta_i$ this property isn't changed. Then
$x=\sum\limits_{s=1}^{a_i-1}\lambda_sv_{\eta_{i-1}+1}\dots
u_{\eta_{i-1}+s}\dots u_{\eta_i}v_{\eta_i+1}\dots
v_{\eta_{i-1}+\psi_i-1}+dz$, $z\cdot v_{\eta_{i-1}+\psi_i}=0$. We
can omit $dz$ according to the third observation.
\begin{multline*}
dv_{\eta_{i-1}+1}\dots u_{\eta_{i-1}+s}\dots
u_{\eta_i}v_{\eta_i+1}\dots
v_{\eta_{i-1}+\psi_i-1}=v_{\eta_{i-1}+1}\dots v_{\eta_{i-1}+s}\dots
u_{\eta_i}v_{\eta_i+1}\dots
v_{\eta_{i-1}+\psi_i-1}-\\
-v_{\eta_{i-1}+1}\dots u_{\eta_{i-1}+s}\dots
v_{\eta_i}v_{\eta_i+1}\dots v_{\eta_{i-1}+\psi_i-1},
\end{multline*}

where both summands are nonzero, so as in (ii) $d\lambda
u_{\eta_{i-1}+1}v_{\eta_{i-1}+2}\dots v_{\eta_{i-1}+\psi_i-1}$
should be equal to $0$ in $\Lambda[u_1,\dots, u_m]\otimes \mathbb
Z[v_1,\dots, v_m]/(v_i^2=u_iv_i=0)$, which is again false. So

{\itshape $\C^{-2,\,*}\cong \mathbb Z^{2k-1}$ with generators
$Y_i=u_{\eta_{i-1}+1}v_{\eta_{i-1}+2}\dots v_{\eta_{i-1}+\psi_i-1}
u_{\eta_{i-1}+\psi_i}$ of graduation $L_i$ corresponding to the
segments $(a_i,\dots,a_{i+k-1})$ of the polygon. $\bideg
Y_i=(-2,2(a_i+\dots+a_{i+k-1}))$.}

\item[{\upshape (iv)}] It is known (see \cite{BP}) that
$\C^{-(m-n),\,2m}\cong\mathbb Z$ with the generator
$u_{\omega}v_{\sigma}$, where $\sigma$ corresponds to an arbitrary
vertex of $P$ and
$\omega=\{1,\dots,m\}\setminus\sigma=[m]\setminus\sigma$. Indeed,
each cochain $u_{\omega}v_{\sigma}$ is a cocycle since any $n+1$
facets have an empty intersection. If two vertices $\sigma_1$ and
$\sigma_2$ are connected by an edge, then
$du_{[m]\setminus(\sigma_1\cap\sigma_2)}v_{\sigma_1\cap\sigma_2}=\pm(u_{\omega_1}v_{\sigma_1}-u_{\omega_2}v_{\sigma_2})$.
Since any two vertices of $P$ are connected by a sequence of edges,
we see that all the cochains $u_{\omega}v_{\sigma}$ represent the
same cohomological class. Since $\mathcal Z_P$ is an oriented
manifold, $\C^{m+n}(\mathcal Z_P)=\mathbb Z$, so the claim is
proved.

In our special case we can provide more details. Let
$\omega=\{i_1,i_2,i_3\}$, where $i_1=1,i_2=\eta_k,i_3=n+3$. These
points correspond to the vertices $a_1,a_k$, and $a_{2k-1}$
respectively. Let $\lambda u_{\omega}v_{\sigma}=dz,\,
\lambda\in\mathbb Z$. Without loss of generality we can take
$z=w\wedge u_{i_1}$. Then $dw=\lambda u_{i_2}\wedge
u_{i_3}v_{\sigma}$ and $w\cdot v_{i_1}=0$. Then each monomial of $w$
is divided by some product
$v_{\eta_{i-1}+1}\dots\widehat{v_{i_1}}\dots
v_{\eta_{i-1}+\varphi_i}$, where $i_1\in J_i$. But $i_2\notin J_i$,
since $i_1$ corresponds to the first vertex of the polygon and $i_2$
-- to the $k$-th. So we can add $dW$ with the property $W\cdot
v_{i_1}=0$ to $w$ to obtain $w'=w+dW$ such that  $w'=\tilde w\wedge
u_{i_2}$

Then $\tilde w\cdot v_{i_1}=\tilde w\cdot v_{i_2}=0$, and $d\tilde
w=-\lambda u_{i_3}v_{\sigma}$. So each monomial of $\tilde w$ is
divided by some $v_{\eta_{i-1}+1}\dots\widehat{v_{i_1}}\dots
v_{\eta_{i-1}+\varphi_i}$ and
$v_{\eta_{j-1}+1}\dots\widehat{v_{i_2}}\dots
v_{\eta_{j-1}+\varphi_j}$, where $i_1\in J_i, i_2\in J_j$. Since
$\tilde w\ne 0$, all the points corresponding to the vertices
$a_{k+1},\dots, a_{2k-1}$ can't belong to the union $J_i\cup J_j$,
so there is some gap $a_s$ in this interval. We claim that the
interval $[1,\dots,\psi_1]$ corresponding to the vertices
$(a_1,\dots, a_k)$ of the polygon belongs to the union $J_i\cup
J_j$. Indeed, either $J_i$ and $J_j$ don't intersect and their union
fills all the circle except for the vertex $a_s$, or they intersect
and cover the segment between $a_1$ and $a_k$. Thus all the
monomials of $\tilde w$ are divided by $v_2\dots v_{\eta_k-1}$.

Then we can add some $d\tilde W,\, \tilde W\cdot v_{i_1}=\tilde
W\cdot v_{i_2}=0$ to obtain $\tilde w'=\tilde w+d\tilde W=w''\wedge
u_{i_3}$ such that $dw''=-v_{\sigma}$, $w''\cdot v_{i_1}=w''\cdot
v_{i_2}=w''\cdot v_{i_3}=0$, $w''$ is divided by $v_2\dots
v_{\eta_k-1}$. But this is impossible, since in this case $w''$
should be divided by $v_1\dots\widehat{v_{i_2}}\dots v_{n+2}$ and
contain no $u$. Thus we proved that {\itshape $\C^{-3,\,2(n+3)}\cong
\mathbb Z$ with a generator $u_{\omega}v_{\sigma}$ for any vertex
$\sigma$.}

Let us note that we also proved the following fact, which we will
use later: {\itshape if $x\cdot v_{i_1},\,x\cdot v_{i_2}=0$ for
$i_1$ and $i_2$ corresponding to the vertices $a_i$ and $a_{i+k-1}$,
and $x\ne 0$, then $x$ is divided by
$v_{\eta_{i-1}+1}\dots\widehat{v_{i_1}}\dots\widehat{v_{i_2}}\dots
v_{\eta_{i-1}+\psi_i}$.}

Let $z$ be an element of graduation $\tau=\{i_1,\dots, i_t\},\,
t<n+3$ and outer degree  $-3$. Then $\tau$ doesn't fill all the
circle, so without loss of generality we can assume that
$i_t+1\notin\tau$. Let $dz=0$. Up to a coboundary $z=w\wedge
u_{i_t}$. Then $dw=0$ and $w\cdot v_{i_t}=0$. So if $w\ne 0$, then
$\tau$ contains $J_i=[\eta_{i-1}+1,\dots,\eta_{i-1}+\varphi_i]$,
$\eta_{i-1}+\varphi_i=i_t$, and $w$ is divided by
$v_{\eta_{i-1}+1}\dots v_{\eta_{i-1}+\varphi_i-1}$. Let $i_1$ be the
index next to $i_t$ on the circle. Then $i_1\ne i_t+1$, and
$i_1\notin J_i$. So we can add $dW,\, W\cdot v_{i_t}=0$ to obtain
$w'=w+dW=w''\wedge u_{i_1}$. Then $w''\cdot v_{i_1}=w''\cdot
v_{i_t}=0$ and $dw''=0$. If $w''\ne 0$, then $w''$ is divided by
$v_{\eta_{j-1}+2}\dots v_{\eta_{j-1}+\varphi_j}$, $\eta_{j-1}+1=i_1$
and by $v_{\eta_{i-1}+1}\dots v_{\eta_{i-1}+\varphi_i-1}$,
$\eta_{i-1}+\varphi_i=i_t$. But each of the intervals $J_i$ and
$J_j$ fills $k-1$ vertices of the polygon, so since
$(\eta_i+1,\dots,\eta_i+a_{i+1})\nsubseteq\tau$, we see that
$J_i\cup J_j=\tau$ and $w''$ contain no $u$. So $w''=0$. Thus we
obtain:

{\itshape $\C^{-3,\,*}=\mathbb Z$ with a generator
$Z=u_{\omega}v_{\sigma}$, where $\sigma=\{i_1,\dots,i_n\}:
F_{i_1}\cap \dots\cap F_{i_n}\ne\varnothing$ is an arbitrary vertex
of the polytope, and $\omega=[m]\setminus\sigma$. $\bideg
Z=(-3,2(n+3))$.}

\item [{\upshape (v)}] Consider an arbitrary element of outer degree
$-4 $ or less. If it's graduation is less than $[m]$, then the
argument as in previous paragraph shows that it is equivalent to
$0$.

Let $x$ be a cocycle of graduation $\tau=[m]=[n+3]$. We can assume
that $x=w\wedge u_1$. If $w\ne 0$, then each monomial of $w$ is
divided by some $v_{\eta_{i-1}}\dots\widehat{v_1}\dots
v_{\eta_{i-1}+\varphi_i}$. We can add $dW$ such that $W\cdot v_1=0$
to obtain : $w+dW=w'\wedge u_{\eta_k}$, $w'\cdot v_1=w'\cdot
v_{\eta_k}=0,\,dw'=0$. If $w'\ne 0$, then as it was mentioned above
$w'$ is divided by $v_2\dots v_{\eta_k-1}$. Then we can add $dW'$
such that $W'$ is divided by $v_2\dots v_{\eta_k-1}$ to obtain
$w'+dW'=w''\wedge u_{n+3}$. If $w''\ne 0$, then $w''\cdot
v_1=w''\cdot v_{\eta_k}=w''\cdot v_{n+3}=0$, $dw''=0$. But in this
case $w''$ should be divided by $v_2\dots\widehat{v_{\eta_k}}\dots
v_{n+2}$, which gives a contradiction, since outer degree of $w''$
is less than $0$.

Thus we have:

$$
\C^{-i,\,*}=0,\quad i>3.
$$
\end{itemize}

At last let us calculate a multiplication. It is easy to see that
for $k\geqslant 3$ we have $X_i\cdot X_j=0$, $Y_i\cdot Y_j=0$,
$X_i\cdot Y_j=\delta_{i+k-1,\,j}Z$. In the case $k=2$ we have
$X_i^2=0$, $X_iX_{i+1}=-X_{i+1}X_i=Y_i$, and $X_1X_2X_3=Z$. In fact,
this case is trivial, since
$P_{a_1,\,a_2,\,a_3}=\Delta^{a_1-1}\times\Delta^{a_2-1}\times\Delta^{a_3-1}$
and $\mathcal Z_{P_{a_1,\,a_2,\,a_3}}=S^{2a_1-1}\times
S^{2a_2-1}\times S^{2a_3-1}$. So
$$
H^{*,\,*}=\mathbb Z[X_1]/(X_1^2)\otimes\mathbb
Z[X_2]/(X_2^2)\otimes\mathbb Z[X_3]/(X_3^2)
$$
Thus we obtain:
\begin{theo*}
For the polytope $P=P_{a_1,\dots,\,a_{2k-1}}$ we have:

The bigraded cohomology ring $\C^{*,\,*}(\mathcal Z_P)$ is a free
abelian group $\mathbb Z\oplus \mathbb Z^{2k-1}\oplus\mathbb
Z^{2k-1}\oplus \mathbb Z$ with the generators
\begin{center}
$1$, $\bideg 1=(0,0)$;\\
$X_i$, $\bideg X_i=(-1,2(a_i+\dots+a_{i+k-2})), i=1,\dots, 2k-1$;\\
$Y_j$, $\bideg Y_j=(-2,2(a_j+\dots+a_{j+k-1})), j=1,\dots,2k-1$;\\
$Z$, $\bideg Z=(-3,2(n+3))$.
\end{center}
For $k\geqslant 3$
$$
X_i\cdot X_j=0,\qquad X_i\cdot Y_j=\delta_{i+k-1,\,j},Z\qquad Y_i\cdot
Y_j=0.
$$
and for $k=2$
$$
X_i^2=0,\qquad X_iX_{i+1}=-X_{i+1}X_i=Y_i,\qquad X_1X_2X_3=Z.
$$
\end{theo*}
\begin{cor}
For simple polytopes $P$ and $Q$ with $n+3$ facets two bigraded
rings $\C^{*,\,*}(\mathcal Z_P,\mathbb Z)$ and $\C^{*,\,*}(\mathcal
Z_P,\mathbb Z)$ are isomorphic if and only if their bigraded Betti
numbers are equal.
\end{cor}

\textbf{Examples}.

\textbf{1.} The polytope $P$ corresponding to the numbers
$(\underbrace{a,a,\dots,a}_{2k-1})$ is a unique combinatorial
polytope with the bigraded cohomology ring $\C^{*,\,*}(\mathcal
Z_P)$.

\textbf{2.} Let $P$ correspond to the sequence $(1,1,2,2,2)$ and $Q$
-- to $(1,1,3,1,2)$. Then
$$
\C^{*,\,*}(\mathcal Z_P)\cong \C^{*,\,*}(\mathcal Z_Q).
$$

\end{document}